\documentclass[11pt,reqno]{amsart}

\usepackage{amssymb}
\usepackage{graphicx}
\usepackage[parfill]{parskip}
\usepackage{hyperref}

\theoremstyle{plain}
\newtheorem{theorem}{Theorem}[section]
\newtheorem{lemma}[theorem]{Lemma}

\theoremstyle{definition}

\theoremstyle{remark}
\newtheorem*{remark}{Remark}

\newtheoremstyle{mytheorem}
  {3pt}
  {3pt}
  {\itshape}
  {}
  {\itshape}
  {.}
  {.5em}
  {}

\theoremstyle{mytheorem}
\newtheorem{question}{Open problem}

\numberwithin{equation}{section}
\numberwithin{theorem}{section}
\numberwithin{table}{section}
\numberwithin{figure}{section}

\title{Nested Recursions, Simultaneous Parameters and Tree Superpositions}

\author[A. Isgur, V. Kuznetsov, M. Rahman, and S. Tanny]{Abraham Isgur \and Vitaly Kuznetsov \and Mustazee Rahman \and Stephen Tanny}

\address[Abraham Isgur, Mustazee Rahman, and Stephen Tanny]{Department of Mathematics\\
University of Toronto\\
40 St. George Street\\
Toronto\\
ON M5S 2E4\\
Canada}

\address[Vitaly Kuznetsov]{Courant Institute of Mathematical Sciences\\
New York University\\
251 Mercer Street\\
New York\\
NY 10012-1185\\
USA}

\email[Abraham Isgur]{umarovi@gmail.com}
\email[Vitaly Kuznetsov]{vitaly@cims.nyu.edu}
\email[Mustazee Rahman]{mustazee.rahman@utoronto.ca}
\email[Stephen Tanny]{tanny@math.utoronto.ca}

\date{\today}



\thanks{Mustazee Rahman's research was supported by a NSERC CGS grant. Vitaly Kuznetsov was partially supported by an Ontario Graduate Scholarship and a NSERC PGS grant.}

\begin{document}

\begin{abstract}
We apply a tree-based methodology to solve new, very broadly defined families of nested recursions of the general form 
$R(n)=\sum_{i=1}^kR(n-a_i-\sum_{j=1}^{p}R(n-b_{ij}))$, where $a_i$ are integers, $b_{ij}$ are natural numbers, and $k,p$ are natural numbers that we use to denote ``arity" and ``order," respectively, and with some specified initial conditions. The key idea of the tree-based solution method is to associate such recursions with infinite labelled trees in a natural way so that the solution to the recursions solves a counting question relating to the corresponding trees. We characterize certain recursion families within $R(n)$ by introducing ``simultaneous parameters" that appear both within the recursion itself and that also specify structural properties of the corresponding tree. First, we extend and unify recently discovered results concerning two families of arity $k=2$, order $p=1$ recursions. Next, we investigate the solution of nested recursion families by taking linear combinations of solution sequence frequencies for simpler nested recursions, which correspond to superpositions of the associated trees; this leads us to identify and solve two new recursion families for arity $k=2$ and general order $p$. Finally, we extend these results to general arity $k>2$. We conclude with several related open problems.

\end{abstract}

\maketitle

\begin{section}{Introduction}\label{sec:intro}

In this paper, all values of the parameters and variables are integers.

Loosely speaking, a nested recurrence relation (also called a meta-Fibonacci recursion) is any recursion where some argument contains a term of the recursion. In a series of recent papers (see \cite{BLT, ConollyLike, Nonhom, Rpaper, JR, DR}) infinite labelled trees are used to solve certain families of nested recursions with the following general form:
\begin{equation}\label{eq:Conolly}
R(n)=\sum_{i=1}^kR(n-a_i-\sum_{j=1}^{p}R(n-b_{ij})),
\end{equation}
where $a_i$ are integers, $b_{ij},k$, and $p$ are natural numbers, and with some specified initial conditions. We call $k$ and $p$ the ``arity" and ``order," respectively, of the recursion, and refer to a recursion with arity $k$ and order $p$ as $k$-ary order $p$. Sometimes we refer to a recursion of the form $(\ref{eq:Conolly})$ as a generalized Conolly-Hofstadter (CH) recursion, for reasons which will become clear below.

A solution to $(\ref{eq:Conolly})$, if it exists, is the (unique) sequence that satisfies the recursion together with its initial conditions. In what follows we often use $R(n)$ or $R$ to refer both to the recursion and its solution, if one exists\footnote{For example, we use $Q(n)$ or $Q$ to refer to Hofstadter's nested recursion, which is defined in \cite{GEB} by $Q(1) = Q(2) = 1$ and $Q(n) = Q(n-Q(n-1)) + Q(n-Q(n-2))$ for $n >2$. $Q$ is a famous example where it is not
known whether or not a solution exists, although the first billion $Q$ recursion values have been computed.}.

We say that a solution sequence is \emph{slowly growing} or \emph{slow} if it has the property that successive differences are either 0 or 1 and it tends to infinity.  Any slowly growing sequence $A(n)$ can be described by its \emph{frequency sequence} $\phi_A(v)$, which counts the number of times that $v>0$ occurs in $A(n)$.

It is evident that the nesting structure of $(\ref{eq:Conolly})$ makes it impossible to apply the usual techniques used for solving (ordinary) difference equations, such as characteristic polynomials and generating functions.\footnote{In some cases (see \cite{ConollyLike, JR} for examples) one can derive generating functions, difference sequences, and frequency sequences for a solution to the nested recursion, but only as a result of prior analysis of the nature of the solution sequence.} Further, except in the simplest cases, there is no explicit or closed form for the solution.

As in \cite{BLT, ConollyLike, Nonhom, Rpaper, JR, DR}, we solve the recursion using our ``tree-based" methodology. By this we mean that we show the existence of an infinite sequence that satisfies the recursion together with its initial conditions, where the $n^{th}$ term of the solution sequence has a counting interpretation in terms of a labelled infinite tree. In this combinatorial interpretation, we have an infinite tree, labelled with integers in preorder, and the solution sequence to the nested recursion counts labels (or some analogue) on the leaves of this tree. It follows that this solution method will naturally identify a slow solution, which is why we restrict ourselves to such solutions in this paper.\footnote{See \cite{IVT} where the tree-based methodology is modified to derive a combinatorial interpretation for a solution sequence with successive differences that are either 0 or $d>1$. Also, in \cite{IR} some nested recursions with slow solutions are studied that do not have a combinatorial interpretation.}

A fundamental contribution of \cite{Rpaper, JR} has been to locate what we call here ``simultaneous parameters."\footnote{In these earlier papers we referred to these parameters as ``shift" parameters. Our new terminology emphasizes the greater generality of these parameters and the dual role they play in both the nested recursion and its corresponding infinite tree.} These are parameters that both appear in the recursion and that also correspond to structural properties of the infinite tree used to derive and interpret its solution. For example, in \cite{Rpaper}, it is shown that the parameters $s \geq 0$ and $j \geq 1$ can be introduced into both the original Conolly recursion (see \cite{Con})
\begin{equation}
\label{eq:0112}
C(n)=C(n-C(n-1))+C(n-1-C(n-2)),\space \space C(1)=1,C(2)=2
\end{equation}
and the original $H$ recursion (see \cite{BLT})
\begin{equation}
\label{eq:0123}
H(n)=H(n-H(n-1))+H(n-2-H(n-3)),\space \space H(1)=H(2)=1,H(3)=2
\end{equation}
to create the more general recursion families
\begin{equation}
\label{eq:0jj2j}
R_{s,j}(n)=R_{s,j}(n-s-R_{s,j}(n-j))+R_{s,j}(n-s-j-R_{s,j}(n-2j))
\end{equation}
and
\begin{equation}\label{eq:0j2j3j}
H_{s,j}(n)=H_{s,j}(n-s-H_{s,j}(n-j))+H_{s,j}(n-s-2j-H_{s,j}(n-3j)).
\end{equation}
Further, and most importantly, it is shown how to alter the labelling of the infinite binary trees corresponding to the solution sequence to $C(n)$ and $H(n)$ respectively, to create new labelled infinite binary trees that correspond to the solution sequences for the more general recursion families including the additional parameters $s$ and $j$.  For example, to derive the solution for  (\ref{eq:0jj2j}), $s$ labels are inserted in the previously empty nodes along the upper spine of the infinite binary tree, and $j$ labels are inserted in each node rather than 1 label per node (see \cite{Rpaper} for a detailed explanation). The other simultaneous parameters $k$ and $p$ in (\ref{eq:Conolly}) are discussed in \cite{DR} and \cite{ConollyLike}.

Identifying simultaneous parameters has proven to be a very powerful way of expanding the range of nested recursions that we can solve. For this reason there is significant interest in finding more such parameters. Once one has been found, the tree methodology offers an effective way to prove how the simultaneous parameter affects the solution to the nested recursion.

The search for new families of nested recursions that can be defined by identifying simultaneous parameters is the starting point for this paper. In Section \ref{sec:order2} we introduce the new simultaneous parameter $m$ into ($\ref{eq:0jj2j}$) and use the tree methodology to solve the resulting order 1 recursion, namely,
\begin{equation}\label{eq:m}
R_{s,j,m}(n)=R_{s,j,m}(n-s-R_{s,j,m}(n-j))+R_{s,j,m}(n-s-j-m-R_{s,j,m}(n-2j-m))
\end{equation}
with $s$ a nonnegative integer, $j$ a natural number, $m$ an integer with $0 \leq m \leq j$, and with appropriate initial conditions.

Observe that when $m=0$ ($\ref{eq:m}$) is identical to $(\ref{eq:0jj2j})$ while when $m=j$ we have $(\ref{eq:0j2j3j})$. Thus, the more general recursion family ($\ref{eq:m}$) contains the above two previously known but seemingly unrelated recursion families as special cases, and also introduces all of the intermediate families of recursions lying ``between" these two previously unconnected recursion families. Thus, by solving (\ref{eq:m}) we are able to unify and extend the results in \cite{Rpaper} in an important way.

In view of this beautiful and unexpected result, it is natural to ask if it is possible to extend other known results about families of nested recursions by combining the simultaneous parameters $k,s,j,m$ and $p$ in interesting ways. For example, recall that in \cite{ConollyLike} the so-called $(\alpha,\beta)$-Conolly recursion of order $p$ is defined as
\begin{equation} \label{eq:alphabeta1}
R(n) = R(n- \sum_{i=1}^p R(n-2i+1)) + R(n-\alpha-\beta - \sum_{i=1}^p R(n-\alpha-\beta-2i+1))
\end{equation}
with $\alpha$ even,  $\beta \geq 0$, $\alpha + \beta \geq 1$ and $p = \alpha/2 + \beta$. With appropriate initial conditions this recursion has a slow, \textit{Conolly-like} solution sequence; that is, its frequency sequence is of the form $\alpha +\beta \phi_C(m)$, where $C$ is the Conolly sequence $(\ref{eq:0112})$. Since it is known (see \cite{Rpaper}) that the frequency sequence of the $H$ sequence ($\ref{eq:0123}$) is the constant sequence 2, it follows that the frequency sequence for the solution to the order $p$ nested recursion (\ref{eq:alphabeta1}) can be written as a linear combination of the frequency sequences to the two simple order 1 nested recurrences $H(n)$ and $C(n)$ defined above. In that sense we can view (\ref{eq:alphabeta1}) as an order $p$ extension of these two latter order 1 recursions.

In Section \ref{sec:orderp} we show how to introduce simultaneous parameters $s$, $j$ and $m$ into ($\ref{eq:alphabeta1}$) in a natural way. Subsequently, we use tree-based solution methods to solve the resulting nested recursion. We identify some interesting analogies between the solution to the extended order $p$ recursion and the order 1 recursion ($\ref{eq:m}$) that also contains these same simultaneous parameters.

As it turns out, however, the solution to the more general order $p$ nested recursion defined in Section \ref{sec:orderp} is not entirely analogous to that for the original $(\alpha,\beta)$-Conolly order 1 recursion. In particular, its frequency sequence \textit{fails} to have the elegant property that it is a linear combination of the frequency sequences of the solutions for ($\ref{eq:0jj2j}$) and ($\ref{eq:0j2j3j}$), which are the $s,j$ extensions to the original $C$ and $H$ recursions.

We address this issue in Section \ref{sec:linearcomb}, where we enhance the tree-based methodology via the notion of \textit{tree superposition} to derive a different 2-ary, order $p$ nested recursion whose solution \emph{does} have the desired property that its frequency sequence is a linear combination of the frequency sequences of the solutions for ($\ref{eq:0jj2j}$) and ($\ref{eq:0j2j3j}$).  In so doing we demonstrate the power of the simultaneous parameter approach to solving nested recursions: we are led to the discovery of the form of this alternate order $p$ recursion through an understanding of the structure of the labelled infinite tree that would be required to provide the desired solution property. This approach is a sort of ``reverse engineering" of the analytical process we have followed to this point, where we have used the tree methodology only to solve a given nested recursion. Once the new recursion is identified in this way, we apply the tree methodology to derive its solution.

In Section \ref{sec:karyorderp} we continue our study of nested recursions via the lens of simultaneous parameters by extending our approach to certain $k$-ary, order $p$ recursion families. In particular we introduce the simultaneous parameter $k \geq 3$ (for the arity of the recursion) into (\ref{eq:alphabeta1}), combining this with the parameter $m$ already introduced in Section \ref{sec:linearcomb}, to yield a new $k$-ary order $p$ family containing $m$. We show that for appropriate choices of $m$ and $p$ the solution of this $k$-ary recursion has frequency sequence $\gamma k + \delta \phi_{C_k}$, where $C_k$ is the solution of the $k$-ary Conolly recursion studied in \cite{DR}.

We conclude in Section \ref{sec:conclusion} with some open questions and comments about future directions for this work.

\end{section}


\begin{section}{Unifying Two Seemingly Unrelated Recursion Families} \label{sec:order2}
This section concerns the process of finding and proving a combinatorial interpretation for the recursion ($\ref{eq:m}$). In so doing we unify and extend the work in \cite{Rpaper} where a tree-based approach is used to solve what appear to be the two unrelated families of nested recursions defined by  $(\ref{eq:0jj2j})$ and $(\ref{eq:0j2j3j})$. We show that in fact these are the natural extreme points of a continuum of families of nested recursions defined by the introduction of the parameter $m$.

Fix $s \geq 0$,  $j \geq 1$, and $m$ with $0\leq m \leq j$. For $m<0$ or $m>j$, the recursions $(\ref{eq:m})$ seem to always be undefined; we will see some heuristic justification for this in the combinatorial interpretation to come. In general, for fixed $s,j,m$ and where there is no confusion, we omit the subscripts and just write $R(n)$.

Define $T = T_{s,j,m}$ to be the following tree. First, draw a skeleton of an infinite binary tree (see Figure \ref{fig:skeleton}). We call the nodes on the extreme left except for the very first node on the bottom left \textit{supernodes} (see the bold boxes in the diagram). All the nodes on the bottom level (including the bottom leftmost node) are called \textit{leaves}, while all other nodes are \textit{regular nodes}. By the left (right) subtree of a node $X$ we mean the left (right) child of $X$ together with all descendants of that child. The leaves that are left (right) children of their parent are called \textit{left (right)} leaves. Nodes on the second level of $T$ are called \textit{penultimate} nodes and are the \textit{parents} of the leaves. Finally, subdivide each leaf into $j$ cells.

\begin{figure}[h]
\includegraphics[scale=.14]{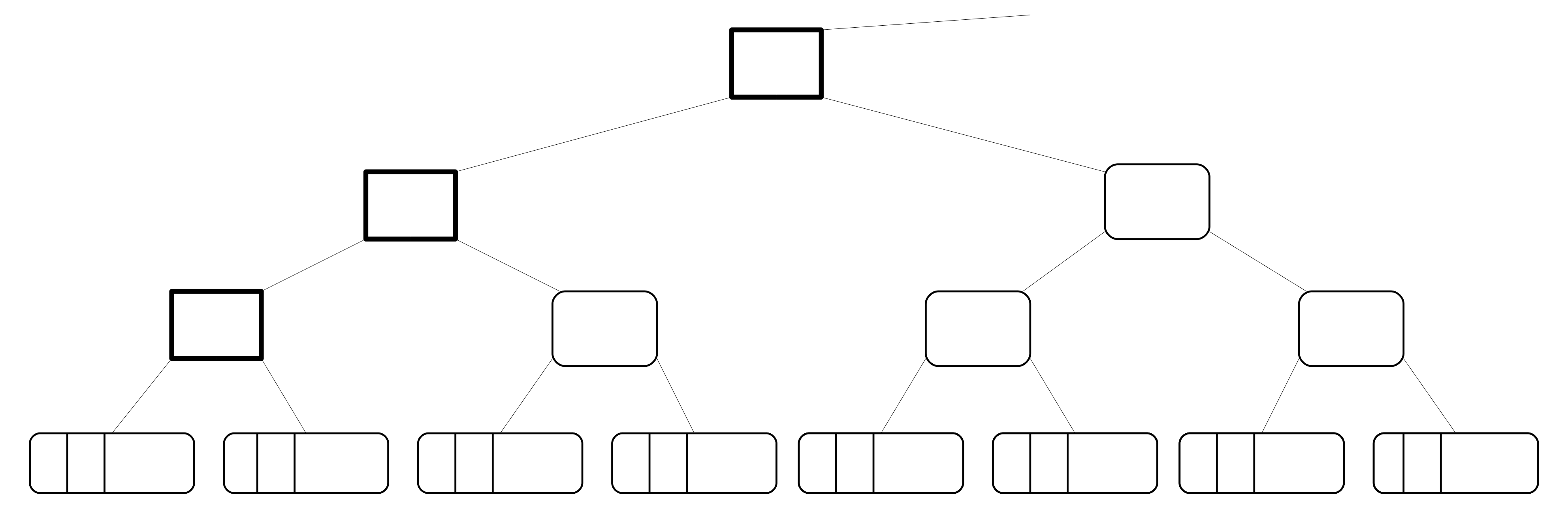}
\caption{Skeleton of an infinite binary tree with $j=3$ cells in each leaf.}\label{fig:skeleton}	
\end{figure}

For each $n \geq 1$ let $T(n)$ denote the infinite tree $T$ with $n$ labels, where these labels are inserted in the nodes of $T$ in preorder as follows: Insert $s$ labels into each supernode, $j-m$ labels into each regular node, and $j+m$ labels into each leaf, placing 1 label in every leaf cell but the last, and $1+m$ labels into the last cell of each leaf. Continue in this way until we have placed $n$ labels in total in preorder. See Figure \ref{fig:r131} for our running example in this section, the tree $T_{1,3,1}(31)$, with $s=1,j=3, m=1$, and $n=31$.

\begin{figure}[h]
\includegraphics[scale=.15]{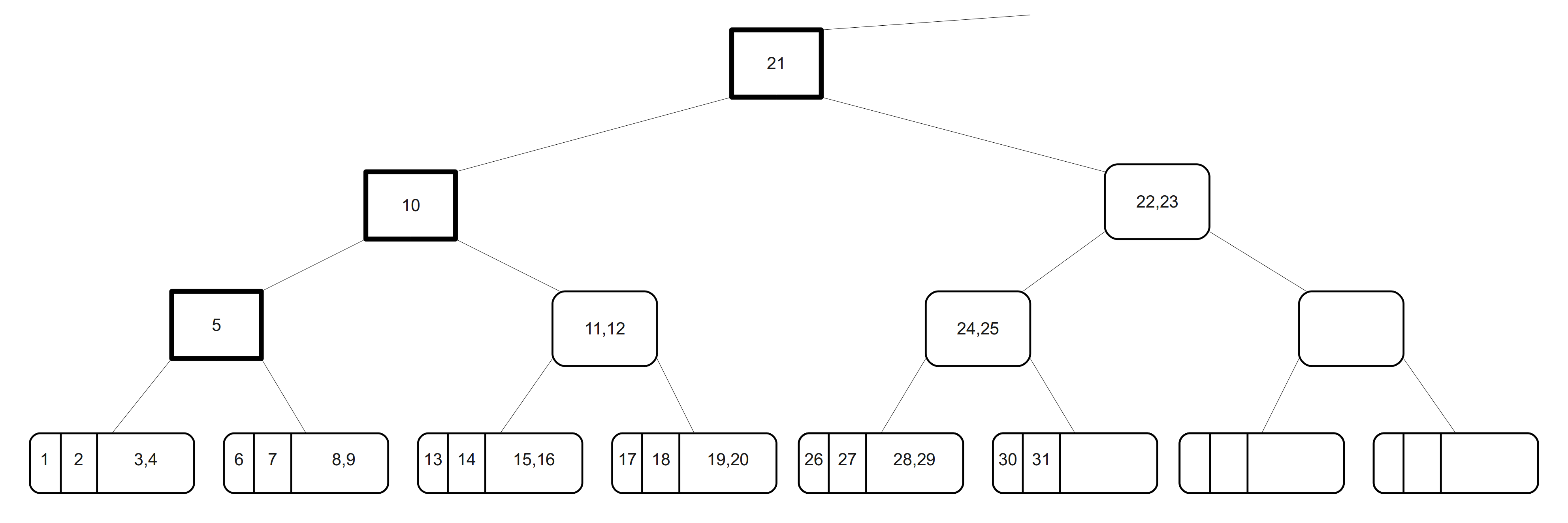}
  \caption{The tree $T_{1,3,1}(31)$, corresponding to the value of $R_{1,3,1}(31)=17$ in the solution sequence to the recursion ($\ref{eq:m}$).}\label{fig:r131}	
\end{figure}

Define the leaf cell counting function $C_T(n)$ to be the number of nonempty cells (that is, cells with at least one label) in the leaves of $T(n)$. In the running example, $C_T(16)=9$. We say that a recursion $R$ with corresponding tree $T$ has initial conditions that \textit{follow the tree} up to $t$ if $R(n)=C_T(n)$ for $1\leq n\leq
t$. For example, for the recursion $R_{1,3,1}(n)$ above, the initial conditions 1,2,3,3,3,4,5,6,6 follow the tree $T_{1,3,1}(n)$ up to $t=9$, which coincides with the last label in the second leaf.

The key result in this section is that the leaf cell counting function satisfies ($\ref{eq:m}$) with sufficiently many initial conditions that follow the tree. More precisely:

\begin{theorem} \label{thm:mmain}
Suppose that the recursion $(\ref{eq:m})$ has initial conditions $R(n)=C_T(n)$ for $n \leq 5j+3m+2s$, that is, the initial conditions follow the tree until the right leaf of the second penultimate level node. Then for all $n$, $R(n)=C_T(n)$, that is, $C_T(n)$ solves the recursion.
\end{theorem}

Notice that this combinatorial interpretation for the solution of $(\ref{eq:m})$ suggests why we cannot allow $m<0$ or $m>j$ in recursion $(\ref{eq:m})$. A tree with $m<0$ would have some leaf cells with either no labels ($m=-1$) or a negative number of labels ($m<-1$), while $m>j$ would lead to negative numbers of labels in the regular nodes.

By definition, $C_T(n)$ is the sum of the number of nonempty cells in the left leaves of $T(n)$ and the number of nonempty cells in the right leaves of $T(n)$. Observe that the number of nonempty cells in the right leaves of $T(n)$ equals the number of nonempty cells in the left leaves of $T(n-j-m)$: to see this, note that if $l$ is a label on a left leaf other than the first leaf in $T(n-j-m)$, then $l+m+j$ is a label on a right leaf in $T(n)$.\footnote{This doesn't hold for the first leaf since the labeling of the first supernode intervenes when $s>0$.} Conversely, if $r$ is a label on a right leaf other than the second leaf of $T(n)$, then $r-m-j$ is a label on a left leaf of $T(n-j-m)$. Since the initial conditions require that we are beyond the first two leaf nodes, which are full and thus have the same number of labels, the fact that this correspondence doesn't hold for the first leaf pair doesn't matter. Thus it follows that there is a one-to-one correspondence between nonempty cells in the left leaves of $T(n-j-m)$ and nonempty cells in the right leaves of $T(n)$.

Therefore, to prove Theorem \ref{thm:mmain} it is enough to show that for $n \geq 4j + 2m + 2s$, the number of nonempty cells in the left leaves of $T(n)$ is $C_T(n-s-C_T(n-j))$. We can then apply this result to the tree $T(n-j-m)$ and use the preceding correspondence to deduce that $C_T(n-s-j-m-C_T(n-2j-m))$ counts the number of nonempty cells in the right leaves of $T(n)$, provided that $n-j-m\geq 4j+2m+2s$, that is, $n\geq 5j+3m+2s$. Adding these cell counts together and combining with the given initial conditions yields the desired solution to the recursion.

In order to prove that for $n \geq 4j + 2m + 2s$ the number of nonempty cells in the left leaves of $T(n)$ is $C_T(n-s-C_T(n-j))$ we define the \textit{pruning operation} for $T(n)$. Note that when $n \geq 4j+2m+2s$ the left leaf of the second penultimate level node necessarily is full.

See Figures $\ref{fig:r131initialcorrection}, \ref{fig:r131deletion}, \ref{fig:r131lifting},$ and $\ref{fig:r131relabelling}$ for an illustration of the pruning process for our running example.

\begin{figure}[h]
 	\includegraphics[scale=.14]{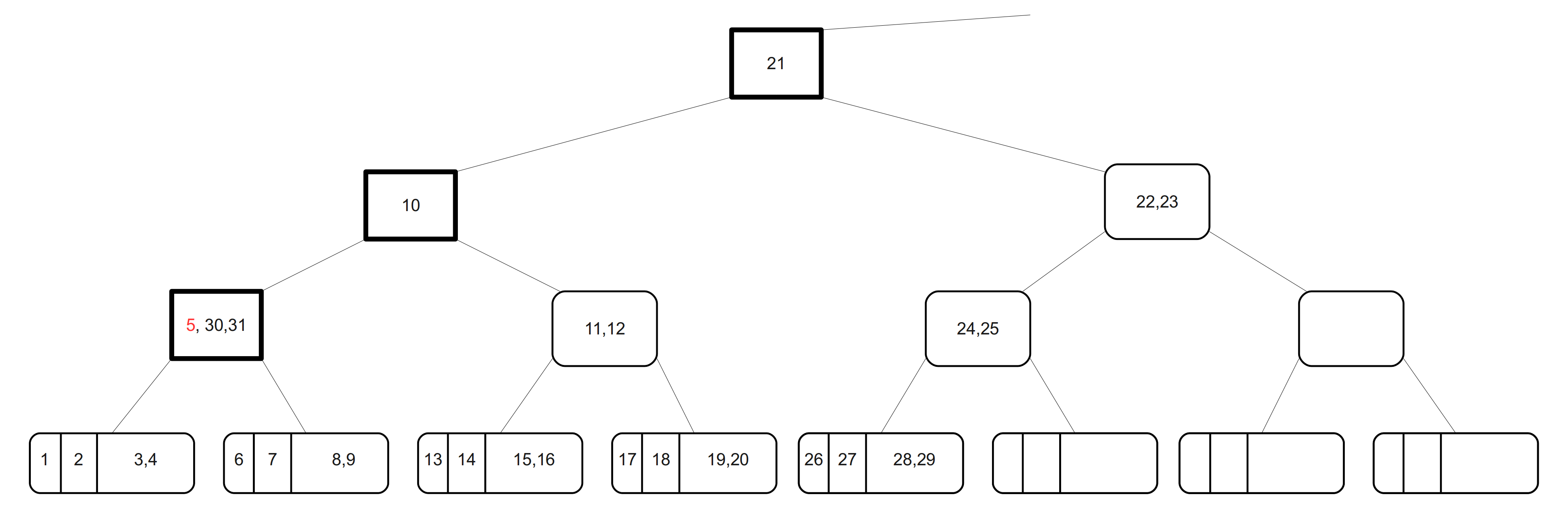}
 \caption{In the initial correction step for $T_{1,3,1}(31)$, the label 5 (in red) is removed from the first supernode and labels 30 and 31 are moved into the first supernode.}\label{fig:r131initialcorrection}
\end{figure}

\begin{figure}[h]
	\includegraphics[scale=.14]{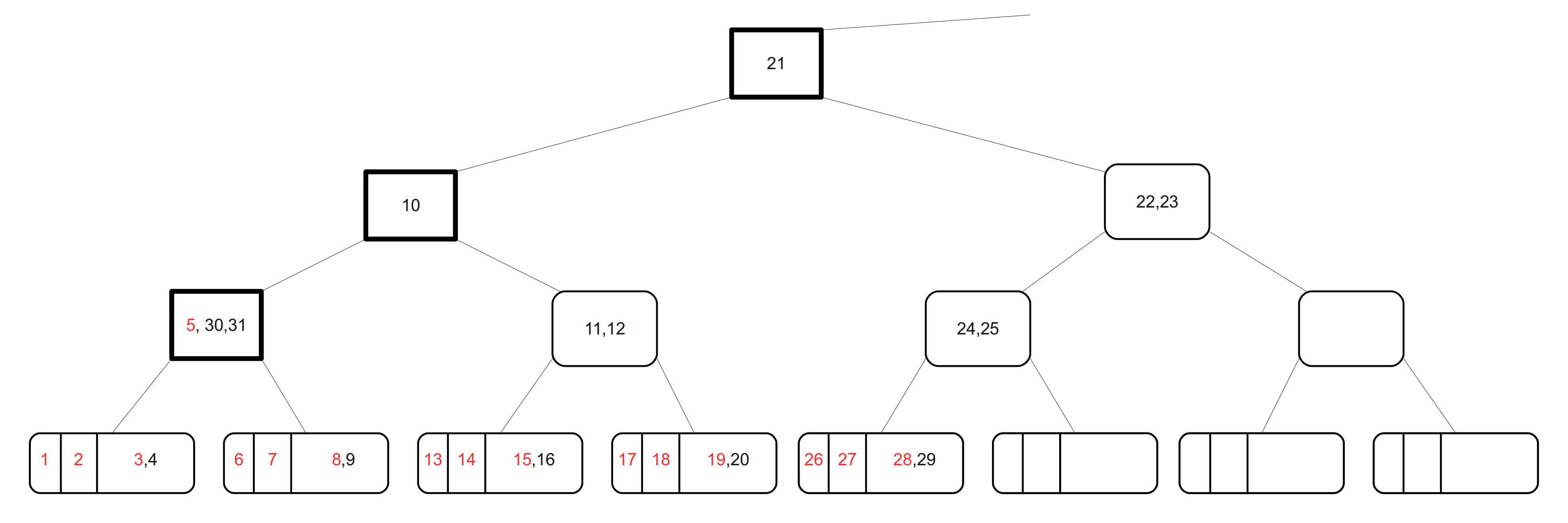}
\caption{In the deletion step for $T_{1,3,1}(31)$, one label (indicated in red) is deleted from every cell that has a label less than or equal to 28.}\label{fig:r131deletion}
\end{figure}

\begin{figure}[h]
	\includegraphics[scale=.14]{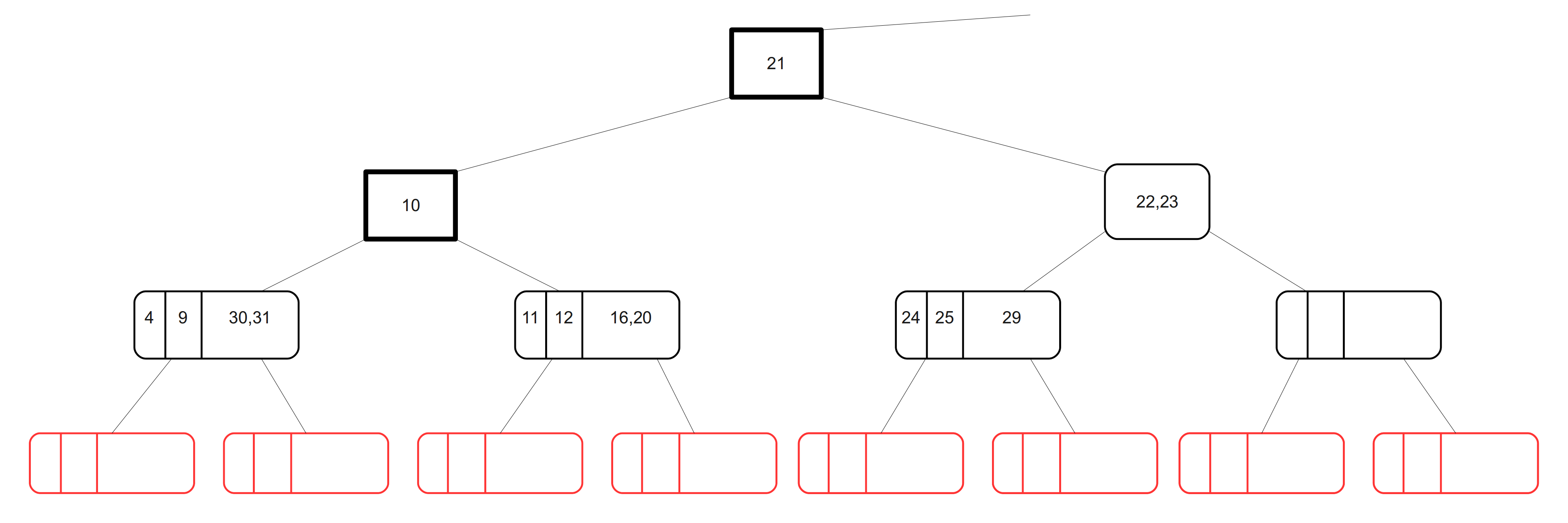}
  \caption{In the lifting step for $T_{1,3,1}(31)$, all the remaining leaf labels move up into their parent penultimate nodes, respectively. The now empty leaves are deleted and the penultimate nodes become the new leaves, with cell divisions introduced and labels inserted according to the rules.}\label{fig:r131lifting}
\end{figure}

\begin{figure}[h]
	\includegraphics[scale=.14]{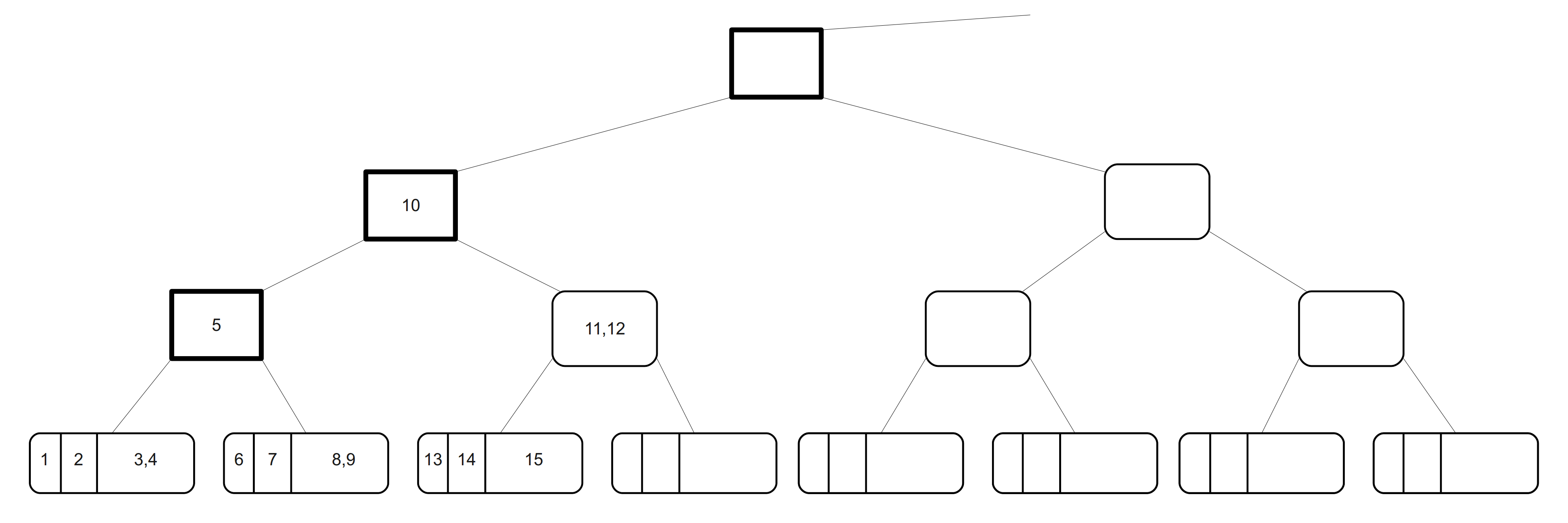}
  \caption{In the relabelling step for $T_{1,3,1}(31)$, the 15 remaining labels are replaced with 1 through 15, showing that $T^*_{1,3,1}(31)=T_{1,3,1}(15).$}\label{fig:r131relabelling}
\end{figure}

\begin{description}

\item[initial correction step]: Delete the $s$ labels in the first supernode (the leftmost penultimate node). Then take the $j-m$ largest labels $n-(j-m)+1,\ldots,n-1,n$ and move them into the now-empty first supernode.

\item[deletion step]: For every cell in $T(n)$, if it has at least one label less than or equal to $n-j$, delete the first label from that cell. This will delete precisely $C_T(n-j)$ labels, by definition of the leaf cell counting function $C_T$. At the end of the deletion step, we have deleted $s+C_T(n-j)$ labels in total.

\item[lifting step]: In all nonempty leaves of our tree (except possibly the last), we will have $m$ labels in the last cell (the last nonempty leaf might have less than $m$ labels in the last cell); this is because we deleted one label from each cell, but the last cell of each leaf had $1+m$ labels, and will thus have $m$ left. Move all remaining leaf labels into the parent of the leaf node they started in.

At this point of the pruning operation, all penultimate nodes (including the first supernode) other than (perhaps) the last nonempty penultimate node have exactly $j+m$ labels (the last nonempty penultimate node may have fewer labels). The first supernode had all of its original $s$ labels removed, $j-m$ labels added from the end, and $m$ labels added from each of its two children. All the other penultimate nodes (except possibly the last nonempty one) started with $j-m$ labels and gained $m$ from each of its two children.

As the last part of the lifting step, convert all the penultimate level nodes into leaves by dividing them into $j$ cells, with one label in each cell but the last, and $1+m$ labels in the last cell of each leaf. The last nonempty penultimate node may not have the $j+m$ labels needed to fill all of its cells, in which case simply fill as many cells as the number of labels on it allows. Finally, delete the bottom level nodes of the current tree (the original leaves), all of which are now empty. This process results in a new tree with the same skeleton as $T(n)$.

\item[relabelling step]: Renumber the labels of the new tree in preorder (so that 1 is the first label, 2 the second, and so on). It is readily seen that the new tree so labelled, which we denote by $T^*(n)$, is identical to $T(n-s-C_T(n-j))$, since it has the same skeleton structure and has $n-s-C_T(n-j)$ labels.

\end{description}

For convenience, we define $C_{T^*}(n)$ to be the number of nonempty leaf cells of $T^*(n)$, that is, we define $C_{T^*}(n)=C_T(n-s-C_T(n-j))$.

We can think of every node in $T^*(n)$ as being part of $T(n)$, that is, we can identify each node in $T^*(n)$ with the node that it was in $T(n)$. For example, we identify the penultimate node with labels 24 and 25 in Figure \ref{fig:r131}  with the leaf node containing labels 13, 14 and 15 in Figure $\ref{fig:r131relabelling}$.

We proceed with a lemma that bijectively relates nonempty cells in the left leaves of the original tree $T(n)$ with nonempty cells in the leaves of the pruned tree $T^*(n)$. In this way we prove that $C_T(n-s-C_T(n-j))$ counts the number of nonempty cells in the left leaves of $T(n)$, and hence
completes the proof of Theorem \ref{thm:mmain}.

\begin{lemma}\label{lem:mstarcount}
Suppose that $P$ is a penultimate node of $T(n)$ (and thus a leaf node in $T^*(n)$), and $n\geq 4j+2m+2s$. Then in $T^*(n)$, the number of nonempty cells of $P$ is equal to the number of nonempty cells of its left child in $T(n)$.
\end{lemma}
\begin{proof}
We begin with the case where $P$ is the first supernode. Note that this is the part of the proof where we rely on the assumption that $n\geq 4j+2m+2s$: the label $4j+2m+2s$ is the last label on the third leaf of $T(n)$, which ensures that the left leaf child of $P$ is full in $T(n)$. Also, we have $n-j\geq 3j+2m+2s \geq 2j+2m+s$ which is the last label on the right child of $P$, so all of the cells of the two children of $P$ will have 1 label removed during the deletion step. Furthermore, note that the second penultimate node (the one just to the right of $P$) has its full complement of $j-m$ labels in $T(n)$ since its last label is $3j+m+2s \leq 4j+2m+2s$. Thus, there are at least $j-m$ labels on the tree after the children of $P$, so the $j-m$ labels moved from the end of the tree into $P$ during the initial correction step will not come from the children of $P$. This means that there remain in place $2m$ labels on the children of $P$ after the deletion step in the pruning process. Therefore in the pruning process the node $P$ will receive the $j-m$ labels from the end of the tree, plus $2m$ labels from its children, making $P$ full in the pruned tree $T^*(n)$ just like its left child is full in $T(n)$. This establishes the required result in this special case.

We now assume that $P$ is not the first supernode. We require several cases:

\textbf{Case 1:} The label $n$ is on a node before (with respect to preorder) the left child of $P$ in $T(n)$, that is, the left child of $P$ has no labels in $T(n)$. We want to show that $P$ has no labels in $T^*(n)$. In this case, observe that $P$ has at most $j-m$ labels before the pruning operation, and during the initial correction step, the final $j-m$ labels in $T(n)$ are moved into the first supernode. This means that any labels in $P$ in $T(n)$ are removed during the pruning, so $P$ will be empty in $T^*(n)$.

\textbf{Case 2:} The label $n$ is one of the first $j$ labels on the left child of $P$ in $T(n)$. Thus the left child of $P$ has $d$ total labels in $T(n)$, where $0<d\leq j$ (observe that there are no labels in the right leaf child of $P$). This means that between 1 and $j$ of the cells of the left child of $P$ have one label each. We want to show that $P$ will have $d$ total labels in $T^*(n)$. Since all of the labels on the left child of $P$ are larger than $n-j$, none of them will be deleted during the deletion step. During the initial correction step, we will move the last $j-m$ labels into the first supernode, and then during the lifting step, the labels remaining on the left child of $P$ (if any) will be moved up into $P$. Since $P$ had $j-m$ labels before the pruning operation, and its children had $d$ labels in total, and (the largest) $j-m$ labels were removed during the initial correction step, there will be $d$ labels left on $P$ after the lifting step of the pruning operation, as desired.

\textbf{Case 3:} The last remaining case is when the $j^{th}$ label on the left child of $P$ in $T(n)$ is smaller than $n$. That is, the left child of $P$ in $T(n)$ has all $j$ of its cells nonempty and $n$ is not the first entry in the last cell. We will show that in $T^*(n)$, the node $P$ also has all of its cells nonempty. To do so we make use of the result we have just proved for Case 2 with $d=j$.

Let $x_j$ be the $j^{th}$ label on the left child of $P$ in $T(n)$. By assumption, we have $x_j<n$. By Case 2, we know that $T^*(x_j)$ has all of the cells of $P$ nonempty. As discussed previously, $T^*(x_j)=T(x_j-s-C_T(x_j-j))$ and $T^*(n)=T(n-s-C_T(n-j))$. If we can prove that $x_j-s-C_T(x_j-j) \leq n-s-C_T(n-j)$, then we will have shown that $T^*(n)$ has at least as many labels as $T^*(x_j)$. Thus $P$ will have at least as many labels in $T^*(n)$ as it does in $T^*(x_j)$, meaning $P$ will have no nonempty cells in $T^*(n)$.

To show that $x_j-s-C_T(x_j-j) \leq n-s-C_T(n-j)$ we will prove that the function $f(n) = n-s-C_T(n-j)$ is monotone nondecreasing. Note that $f(n+1)=n+1-s-C_T(n+1-j)$. Since $C_T$ counts nonempty leaf cells, either $C_T(n+1-j)=C_T(n-j)$ (if $n+1-j$ is not the first label of a leaf cell), or $C_T(n+1-j)=C_T(n-j)+1$ (if $n+1-j$ is the first label of a leaf cell). In the former case, we have $n+1-s-C_T(n+1-j)=n-s-C_T(n-j)+1$, and in the latter case, we have $n+1-s-C_T(n+1-j)=n-s-C_T(n-j)$. Either way, this establishes that $f(n+1) \geq f(n)$, proving the desired inequality. This completes the proof of Case 3 and the lemma, so Theorem \ref{thm:mmain} is established.
\end{proof}

As we pointed out earlier, the introduction of the simultaneous parameter $m$ in $(\ref{eq:0jj2j})$ defines a new recursion family that unifies the results about $(\ref{eq:0jj2j})$ and $(\ref{eq:0j2j3j})$ proved in \cite{Rpaper}: when $m=0$ ($\ref{eq:m}$) is identical to $(\ref{eq:0jj2j})$ while when $m=j$ we have $(\ref{eq:0j2j3j})$. Thus, we have solved these two previously known recursions and also all of the intermediate recursions lying between them. In addition, the parameter $m$ plays a key role in the structure of the resulting tree used in solving ($\ref{eq:m}$); it is also noteworthy that through the generality that $m$ provides the derivation of the solution to ($\ref{eq:m}$) is even easier than the solutions to $(\ref{eq:0jj2j})$ and $(\ref{eq:0j2j3j})$ in \cite{Rpaper}.

We conclude this section by deriving the frequency sequence of the solution that we have just determined to ($\ref{eq:m}$).

\begin{theorem}\label{thm:mfreq}
The solution $C_T$ to the nested recursion $(\ref{eq:m})$ has frequency sequence

\begin{displaymath}
\phi_{C_T}(v) = \left \{
\begin{array}{lr}
1& \text{if}\; j \nmid v\\
(j-m) \nu_2 (v/j) + m + 1 + s \mathbf{1}_{[\frac{v}{j} \; \text{is a power of 2}]}   & \text{otherwise} \end{array}
\right.
\end{displaymath}

where $\nu_2(x)$ is the 2-adic valuation of $x$ and $\mathbf{1}_{[E]}$ is the indicator function of the set $E$.
\end{theorem}

\begin{proof}

We begin by counting the number of regular nodes between the $h^{th}$ and $(h+1)^{st}$ leaves in $T$. We consider two cases.

Suppose $h$ is a power of $2$, say $h=2^b$. Then observe that the $h^{th}$ leaf comes right before the $(b+1)^{st}$ supernode. By the tree construction this supernode is the root of a complete binary subtree of height $(b+1)$, so it is followed in turn in preorder by $b$ regular nodes (the ``leftmost" nodes of the complete binary subtree rooted at the $(b+1)^{st}$ supernode), and then by the $(h+1)^{st}$ leaf. Thus, the number of regular nodes between $h^{th}$ and $(h+1)^{st}$ leaves is $\nu_2(h)$.

Assume that $h$ is not a power of $2$, say $h = a 2^b$ for some odd integer $a > 1$. Consider the node $N$ in $T$ such that the $h^{th}$ leaf is the $(2^{b+1} - 1)^{st}$ node in preorder following $N$. For example, in Figure \ref{fig:r131initialcorrection}, for $h = 6$ the node $N$ contains the labels 22 and 23. Recall that $(2^{b+1} - 1)$ is the number of nodes in a complete rooted binary tree of height $b$. It follows that the $h^{th}$ leaf is the rightmost node in the left subtree of $N$ and the $(h+1)^{st}$ leaf is the leftmost node in the right subtree of $N$ and there are $\nu_2(h) = b$ regular nodes in preorder between $h^{th}$ and $(h+1)^{st}$ leaves.\footnote{An alternate approach to counting the number of nodes between the $h^{th}$ and $(h+1)^{st}$ leaves in $T$ is as follows (see \cite{Isgur}): in \cite{JR} it is shown that the Conolly sequence is the label count on a binary tree with empty supernodes. From this it follows that the frequency with which $h$ occurs in the Conolly sequence is precisely the number of regular nodes between the $h^{th}$ and $(h+1)^{st}$ leaves, plus one. Further, it is also shown in \cite{JR} that the frequency of the Conolly sequence is $\nu_2(v) + 1$. Therefore, the number of regular nodes between the $h^{th}$ and $(h+1)^{st}$ leaves of a binary tree is $\nu_2(h)$.}

Now we proceed with the proof of the theorem. If $v$ is not a multiple of $j$, then the $v^{th}$ nonempty cell is not the last cell on a leaf. Since cells other than the last cell on a leaf have one label and are followed by another cell, we have $\phi_{C_T}(v)=1$ in this case. If $v$ is a multiple of $j$, then the $v^{th}$ nonempty cell is the last cell on the $(v/j)^{th}$ leaf. Thus, the value $v$ is assumed on all $1+m$ labels in that cell, plus the $\nu_2(v/j)$ regular nodes following it (each with $j-m$ labels), plus another $s$ labels on a supernode if $v/j$ is a power of $2$ (and hence the $(v/j)^{th}$ leaf comes right before a supernode). This establishes the stated values for the frequency sequence for the solution to recursion $(\ref{eq:m})$.

\end{proof}

The above argument is a general technique for deriving the frequency sequence for a solution related to labelled trees of this type. This proof technique does not depend on the labelling scheme, but only on the skeleton of the tree, so it can be adapted to other situations such as the one we will discuss in the following section (see Section \ref{sec3:frequency}). 

\end{section}


\begin{section}{Simultaneous parameters in higher order nested recursions} \label{sec:orderp}
In this section we apply our tree-based approach to solve higher order recursions containing simultaneous parameters. Our starting point is the 2-ary, order $p$ recursion (\ref{eq:alphabeta1}) that first appears in \cite{ConollyLike}. As is discussed there, this recursion can be viewed as an order $p$ extension of both the order 1 Conolly recursion ($\ref{eq:0112}$) (take $\alpha=0$ and $\beta=1$) and the $H$ recursion ($\ref{eq:0123}$) (take $\alpha=2$ and $\beta=0$). Here we show how to construct and solve a natural extension to this recursion that contains simultaneous parameters $s,j$ and $m$ that each play a role analogous to the one they played in ($\ref{eq:m}$).  In this way we derive a tree-based interpretation for the solution to a higher order generalization of (\ref{eq:m}). We explore similarities in the behaviour of this solution and those for ($\ref{eq:m}$) and (\ref{eq:alphabeta1}), respectively.

In a manner formally similar to the approach taken in \cite{Rpaper} for the simultaneous parameters $s$ and $j$ and in the preceding section for the parameter $m$, we introduce what we will see are simultaneous parameters $s,j$ and $\bar{m}$ into (\ref{eq:alphabeta1}) in what appears to be a very natural way, namely:
\begin{equation} \label{rec2}
R(n) = R(n- s - \sum_{i=1}^p R(n-(2i-1)j)) + R(n-s-(\alpha+\beta)j -\bar{m} - \sum_{i=1}^p R(n-(\alpha+\beta)j-\bar{m}-(2i-1)j)),
\end{equation}
with, as in (\ref{eq:alphabeta1}), $\alpha$ even,  $\beta \geq 0$, $\alpha + \beta \geq 1$, $p = \alpha/2 + \beta \geq 1$, $s \geq 0$, $j \geq 1$ and some range of values for $\bar{m}$ that we discuss below.

Some modest experimentation with particular values for the parameters in ($\ref{rec2}$) is sufficient to demonstrate that this parametrization in terms of $(\alpha, \beta, \bar{m})$ is not one to one, that is, different choices of parameters $(\alpha, \beta, \bar{m})$ can generate the same recursion and through it the same tree and solution sequence. For example, $s=0, j=3, \alpha=2, \beta=1, \bar{m}=1$ and $s=0, j=3, \alpha=-2, \beta=3, \bar{m}=7$ generate the same order $p=2$ recursion and associated tree.

This duplication occurs because $\alpha, \beta$ and $\bar{m}$ always appear together in ($\ref{rec2}$). To eliminate this duplication we combine them into a single new parameter $m = (\alpha+\beta -1)j + \bar{m}$. As a result, ($\ref{rec2}$) becomes
\begin{equation} \label{rec1}
R(n) = R(n - s - \sum_{i=1}^p R(n-(2i-1)j)) + R(n-s-j-m- \sum_{i=1}^p R(n-j-m-(2i-1)j)).
\end{equation}
Note that (\ref{rec1}) reduces to (\ref{eq:m}) when $p=1$, so the results of this section generalize those of  Section \ref{sec:order2}.

Computational evidence to date with (\ref{rec1}) suggests that whenever this recursion generates an infinite solution sequence (for some set of initial conditions) then $0 \leq m \leq (2p-1)j$. For this reason we restrict $m$ to this range.\footnote{In fact, this is the range of $m$ for which our tree-based proof below holds. This suggests a heuristic reason for our inability to locate any solutions for the recursion with $m$ outside this range.}

\subsection{Construction of the tree and statement of the main theorem} \label{sec3:mainresult}

The skeleton of the tree $T = T_{s,j,m,p}$ that we use here is the same infinite binary tree as in Section \ref{sec:order2}, and we adopt the same terminology and a similar labelling scheme. For $n \geq 1$ let $T(n)$ denote $T$ with the first $n$ labels inserted in preorder according to the following rules: the supernodes of $T$ contain $s$ labels each, every leaf node of $T$
contains $j$ cells with 1 label in each of the first $j-1$ cells and $1+m$ labels in the last cell, and all other regular nodes contain $x := (2p-1)j - m$ labels each. Continue in this way until we have placed $n$ labels. Figure \ref{fig:T63} shows $T(63)$ for our running example in this section with $s=0,j=3$ and $m=p=2$. Note that in this case $x=7$.

\begin{figure}[h]
\includegraphics[scale=.14]{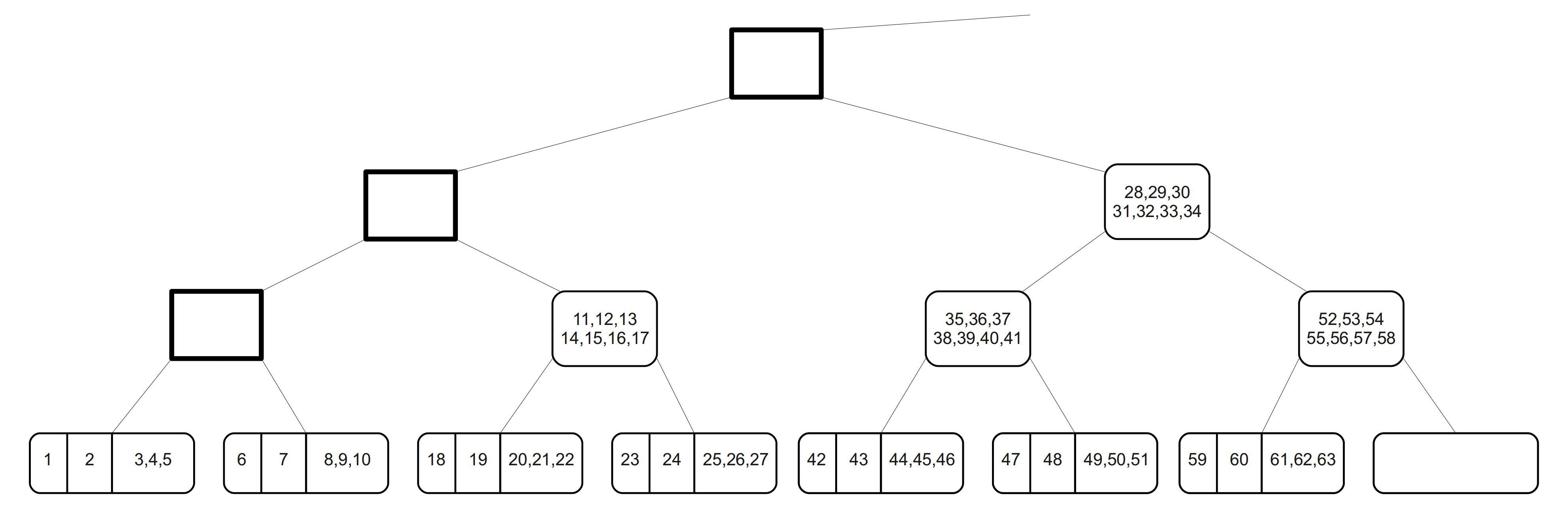}
\caption{The labeled tree $T(63)$ for $(s,j,m,p) = (0,3,2,2); C_T(63)=21$.}\label{fig:T63}	
\end{figure}

Define the leaf cell counting function $C_T(n)$ to be the number of nonempty cells in $T(n)$.
The main result of this section is that $C_T(n)$ satisfies (\ref{rec1}) with appropriate initial conditions that follow the tree $T$.

\begin{theorem} \label{claim1} Suppose that the recursion $(\ref{rec1})$ has initial conditions $R(n)=C_T(n)$ for
$n \leq 4(j+m) + x + 2s$, that is, the initial conditions follow the tree until the right leaf of the second penultimate level node. Then for all $n$, $R(n)=C_T(n)$.
\end{theorem}

Before we prove Theorem \ref{claim1} we examine the special endpoint cases $m=0$ and $m=(2p-1)j$ associated with the range of $m$. When $m=0$ then (\ref{rec1}) is an order $p$ analogue of (\ref{eq:0jj2j}), while if $m=(2p-1)j$ then (\ref{rec1}) is an order $p$ analogue of (\ref{eq:0j2j3j}). In particular, when $s=0$ and $j=1$ the first recursion with $m=0$ is an order $p$ analogue to the Conolly recursion ($\ref{eq:0112}$) while the second recursion with $m=(2p-1)j$ is an order $p$ analogue to the $H$ recursion ($\ref{eq:0123}$) with solution $\lceil n/2 \rceil$ (in both cases, the required initial conditions are generated from the associated tree).

Even further, as in the case of ($\ref{eq:0123}$), it will be evident from what follows that the solution to (\ref{rec1}) with $m=(2p-1)j$, $s=0$ and $j=1$ is a ceiling function, in this case $\lceil n/2p \rceil$. While a more general version of this result appears in  \cite{ConollyLike}, we provide here the first tree-based derivation of a ceiling function solution for an order $p$ nested recursion.\footnote{See \cite{ConollyLike}, Theorem 5.2, where a very different methodology is used to characterize all recursions of the form $R(n) = R( n-s-\sum_{i=1}^p R(n-a_i) ) + R(n-t-\sum_{i=1}^p R(n-b_i))$ with solution $\lceil n/2p \rceil$. We believe that our proof will work for any of the recursions stated in Theorem 5.2 of \cite{ConollyLike} provided that
$s=0, a_i < 2p, t = 2p$ and $b_i = a_i + 2p$.}

\subsection{Strategy of the proof : the pruning operation} \label{sec3:pruning}

We will follow a similar approach to that adopted in Section \ref{sec:order2}. To prove Theorem \ref{claim1}, first we denote by $C_{T,L}(n)$ and $C_{T,R}(n)$ the number of nonempty cells in $T(n)$ that are on the left and right leaves, respectively. By definition
$$ C_T(n) = C_{T,L}(n) + C_{T,R}(n).$$

Since there are $j+m$ labels in total in a full leaf, there is a natural bijection between the nonempty cells of $T(n)$ that are on right leaves and the nonempty cells of $T(n-j-m)$ that are on left leaves. Thus,
$$ C_{T,R}(n) = C_{T,L}(n-j-m).$$

Hence, to prove Theorem \ref{claim1} it suffices to show the following result:
\begin{lemma} \label{claim2}
For $n > 3(j+m) + x + 2s$, we have that $$C_{T,L}(n) = C_T(n-s-\sum_{i=1}^p C_T(n-(2i-1)j))\,.$$\end{lemma}

As in Section \ref{sec:order2}, our proof relies on a pruning technique on $T(n)$ that we now describe. Once again we use $T^*(n)$ to denote the pruned tree that results from applying this technique. See Figures \ref{fig:s=0_j=3_y=2_p=2_initialcorrection}, \ref{fig:s=0_j=3_y=2_p=2_deletion}, \ref{fig:s=0_j=3_y=2_p=2_lifting}, \ref{fig:s=0_j=3_y=2_p=2_endcorrection} and \ref{fig:s=0_j=3_y=2_p=2_afterpruning} where we illustrate the pruning operation on our running example.

\begin{figure}[h]
 	\includegraphics[scale=.14]{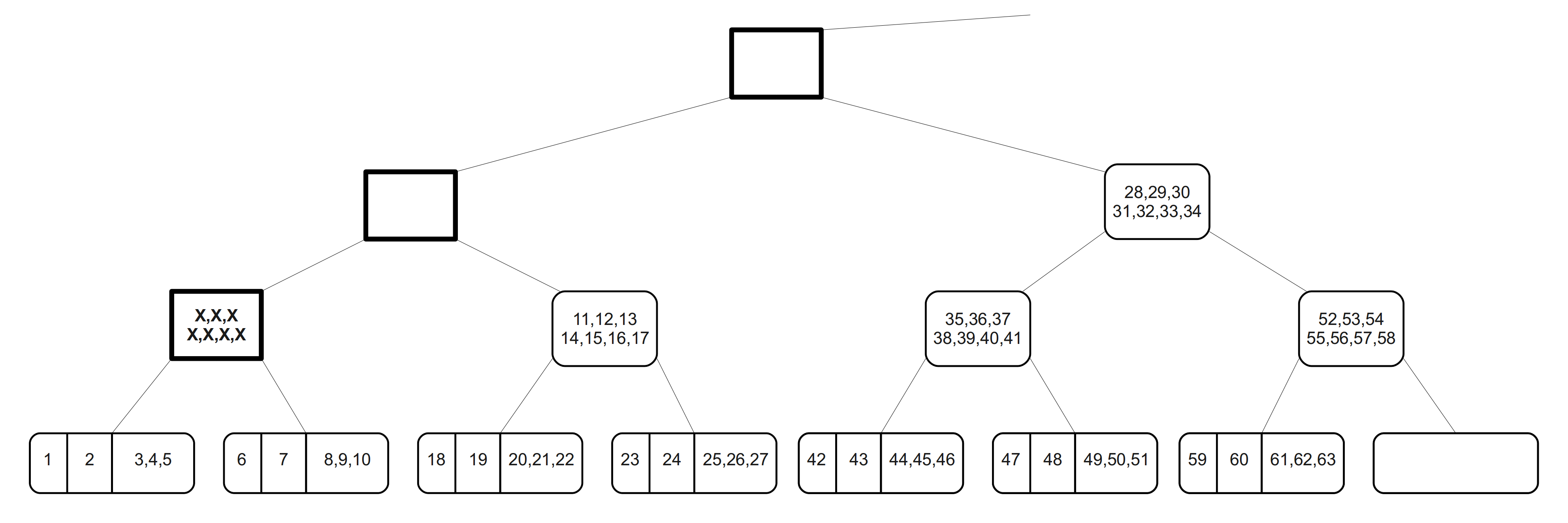}
 \caption{Initial correction step when pruning $T(63)$ where $(s,j,m,p) = (0,3,2,2)$.}\label{fig:s=0_j=3_y=2_p=2_initialcorrection}
\end{figure}

\begin{figure}[h]
	\includegraphics[scale=.14]{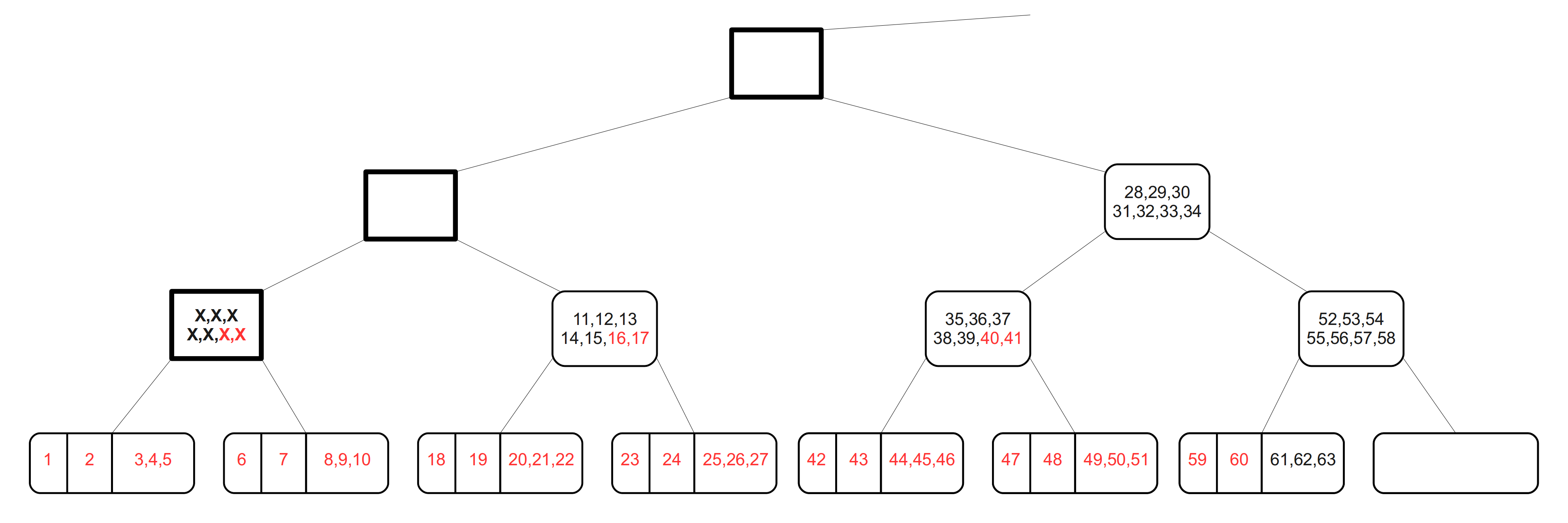}
\caption{Deletion step when pruning $T(63)$ where $(s,j,m,p) = (0,3,2,2)$.}\label{fig:s=0_j=3_y=2_p=2_deletion}
\end{figure}

\begin{figure}[h]
	\includegraphics[scale=.14]{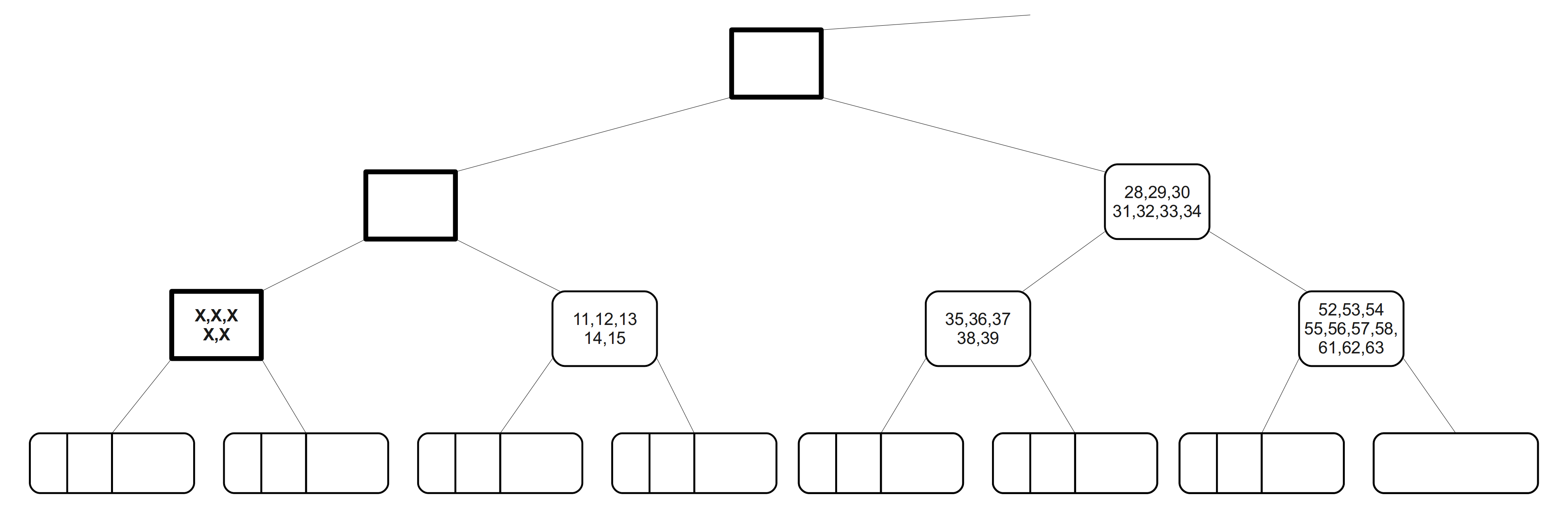}
  \caption{Lifting step when pruning $T(63)$ where $(s,j,m,p) = (0,3,2,2)$.}\label{fig:s=0_j=3_y=2_p=2_lifting}
\end{figure}

\begin{figure}[h]
	\includegraphics[scale=.14]{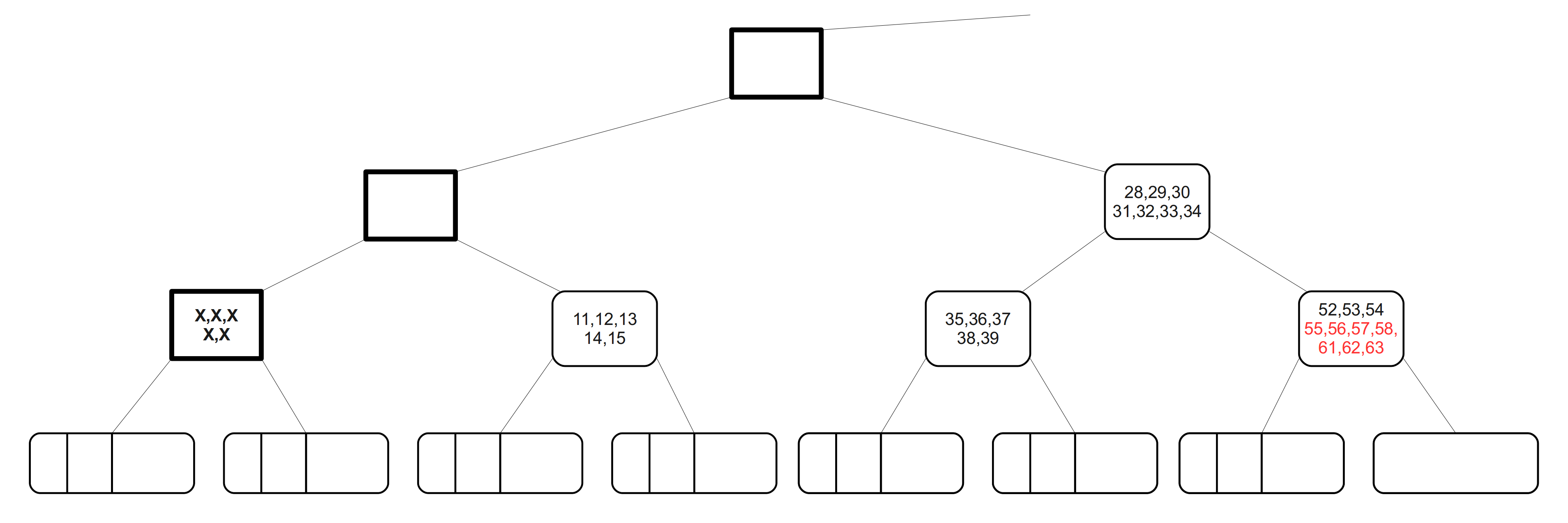}
  \caption{End correction when pruning $T(63)$ where $(s,j,m,p) = (0,3,2,2)$.}\label{fig:s=0_j=3_y=2_p=2_endcorrection}
\end{figure}

\begin{figure}[h]
	\includegraphics[scale=.14]{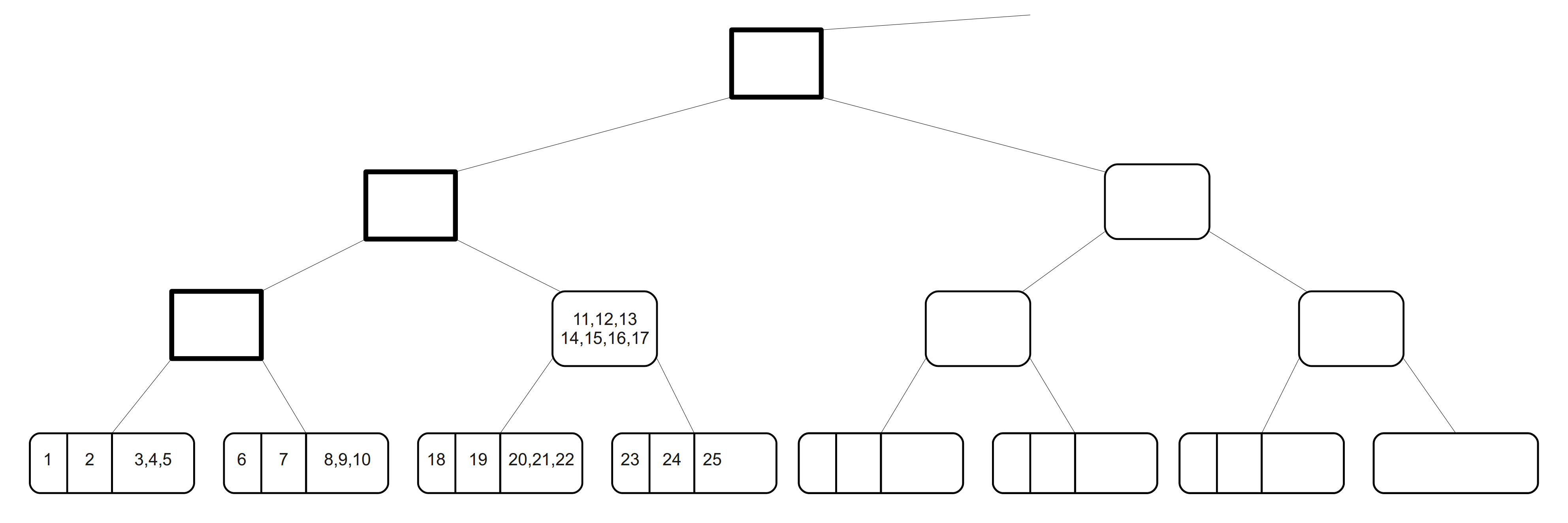}
  \caption{Relabelling step when pruning $T(63)$ where $(s,j,m,p) = (0,3,2,2)$.}\label{fig:s=0_j=3_y=2_p=2_afterpruning}
\end{figure}

\begin{description}
\item[initial correction step] Remove the $s$ labels from the first supernode of $T(n)$ and replace them with $x$ labels. We do not
identify these new labels until we reach the relabelling step below.
\item[deletion step] For each $i = 1, \ldots, p$, consider the tree $T(n-(2i-1)j)$ as a subtree of $T(n)$. For every nonempty
cell in $T(n-(2i-1)j)$ remove a label from the corresponding cell in $T(n)$. If that cell in $T(n)$ has already been emptied by an earlier application of this process then remove a label from the last cell of the corresponding leaf containing the empty cell (so long as a label is available). If both the cell
and the last cell of the leaf containing said cell already have been emptied by this process, then remove a label from the corresponding parent node at the
penultimate level.

We pause to confirm that there are enough labels to apply the instructions in this step. Note that each leaf has $j$ cells and there are $p$ subtrees so the maximum number of labels that can be attempted to be deleted from a pair of sibling leaves is $2pj$. But the total number of labels
within such a pair of sibling leaves and their parent is $2(j+m) + x = m + (2p+1)j>2pj$ so there are enough labels.
\item[lifting step] Lift any remaining labels in a leaf into the corresponding parent at the penultimate level. Note that any leaf with the first $j$ of its labels all less than $n-(2p-1)j$ will be left with $m-(p-1)j$ labels to lift so long as $m>(p-1)j$, and 0 labels otherwise.
\item[end correction step] Remove the $x$ largest labels in preorder of $T(n)$ that remain after the lifting step.
Note that this removal fully offsets the insertion of $x$ labels in the first supernode during the initial correction step.
\item[relabelling step] Remove the leaves of $T(n)$ (which are now empty) and relabel the new tree in preorder.
The former penultimate level nodes are now the leaves of the new tree. Make $j$ cells for every leaf and assign 1 label in each of the first $j-1$ cells and the remaining entries in the last cell. If some leaf has less than $j+m$ labels
then fill in its cells as just described, recognizing that some cells may remain empty or, in the case of the last cell, only partially filled.
\end{description}

In pruning $T(n)$ we have removed $s + \sum_{i=1}^p C_T(n-(2i-1)j)$ labels. Thus $T^*(n)$ is a tree with the same skeleton structure as $T(n)$ and with $n-s-\sum_{i=1}^p C_T(n-(2i-1)j)$ labels.

The strategy behind our proof of Lemma \ref{claim2} is to show first that the pruning operation on $T(n)$ creates a tree $T^*(n)$ that conforms to the labelling rules that we described above. Together with what we have just observed about the skeleton of the tree $T^*(n)$, this will imply that $T^*(n)$ is identical to $T(n-s-\sum_{i=1}^p C_T(n-(2i-1)j))$, so that the number of nonempty cells in $T^*(n)$ is $C_T(n-s-\sum_{i=1}^p C_T(n-(2i-1)j))$.
Then we will demonstrate a bijection between the nonempty cells of $T^*(n)$ and the nonempty cells of $T(n)$ that are on left leaves.
Together these two assertions imply Lemma \ref{claim2}. Thus, to prove Lemma \ref{claim2} we will establish the following two lemmas:

\begin{lemma} \label{claim3}
The pruned tree $T^*(n)$ is identical to  $T(n-s-\sum_{i=1}^p C_T(n-(2i-1)j)$.
\end{lemma}

\begin{lemma} \label{claim4}
Let $P$ be a penultimate node of $T(n)$. Then $P$ becomes a leaf node of $T^*(n)$ and the number of nonempty cells of $P$ in $T^*(n)$ is equal to the number of nonempty cells of the left child of $P$ in $T(n)$.
\end{lemma}

\subsection{Proof of Lemmas \ref{claim3} and \ref{claim4}.}\label{sec3:proof}

We prove both lemmas simultaneously. We begin with a preliminary discussion of each.

As noted above, to prove Lemma \ref{claim3} we need only show that the labelling of $T^*(n)$ is in accordance with the rules that we have laid out above. That is, except for the last nonempty node in $T^*(n)$, all super nodes of $T^*(n)$ have $s$ labels, the leaves have $j+m$ labels and the regular nodes have $x$ labels each. Finally, the last nonempty node of $T^*(n)$ cannot contain more than $s,j+m$ or $x$ labels respectively, depending on its type.

As in Section \ref{sec:order2}, we can think of every node in $T^*(n)$ as being part of $T(n)$. By the design of the pruning operation on $T(n)$, all nodes in $T^*(n)$, except for the leaves of $T^*(n)$ (which are the former penultimate nodes of $T(n)$) and the last nonempty node (which may or may not be a leaf of $T^*(n)$), contain the same number of labels as they do in $T(n)$. So to prove Lemma \ref{claim3} we need only focus on the leaves of $T^*(n)$ and its last nonempty node. First, we make a simple yet important observation that is used several times in the argument.

\begin{lemma} \label{lem2}
Let $\mathbf{P}$ be a penultimate level node of $T(n)$ with left child $\mathbf{L}$ and right child $\mathbf{R}$. Suppose that $T(n)$ contains at least one completely filled regular node following $\mathbf{R}$. Then as a leaf of $T^*(n)$, the node $\mathbf{P}$ contains $j+m$ labels.
\end{lemma}

\begin{proof}
The node $\mathbf{R}$ has $j+m$ labels on itself in $T(n)$. By assumption, the first regular node in $T(n)$ that follows $\mathbf{R}$, say $Q$, is full with $x$ labels on it (note that it cannot be a leaf). Therefore, there are at least $j+m+x=2pj$ labels on or after $\mathbf{R}$ in $T(n)$. Consequently, $\mathbf{R}$ has at least $j$ labels (and thus all $j$ nonempty cells) in each of the subtrees $T(n-j), \ldots, T(n-(2p-1)j)$, and thus so does $\mathbf{L}$. We conclude that the total number of labels removed from $\mathbf{P}$, $\mathbf{L}$, and $\mathbf{R}$ during the deletion step is the maximum amount, namely $2pj$. Note that since $Q$ has $x$ labels, and comes after $\mathbf{R}$, none of the labels removed for the end correction step will come from $\mathbf{P}$, $\mathbf{L}$, or $\mathbf{R}$. Thus, after the lifting step moves all remaining labels from $\mathbf{L}$ and $\mathbf{R}$ to $\mathbf{P}$, and pruning is completed, $\mathbf{P}$ will have $x + 2(j+m) - 2pj = j + m$ labels.
\end{proof}

Going back to the discussion about Lemma \ref{claim3}, consider first the last nonempty node of $T^*(n)$. Suppose that it is \emph{not} a leaf of $T^*(n)$. Then in $T(n)$ this node is neither a leaf nor a penultimate node. Therefore the pruning operation on $T(n)$ doesn't add any labels to this node (the end correction step of the pruning may remove some labels). After pruning, this last nonempty node in $T^*(n)$ has at most the same number of labels that it has in $T(n)$, which is what we require.

Next we turn to the leaves of $T^*(n)$. First consider a leaf $P$ of $T^*(n)$ that is \emph{not} one of the last two nonempty penultimate nodes in $T(n)$. Then, in $T(n)$, $P$ is a penultimate level node with its full complement of $x$ labels, both its children must contain a full complement of $j+m$ labels, and there must be a full penultimate node with two full leaf children that follow $P$. We can thus apply Lemma \ref{lem2} to conclude that the number of labels in $P$ in $T^*(n)$ is $j+m$, as required.

Finally we consider the two leaves $P_a$ and $P_b$ of $T^*(n)$ that are the last two nonempty penultimate nodes in $T(n)$, where $P_b$ is to the right of $P_a$. To establish the required result for these two nodes we have to show: (1) if $P_b$ is not the last nonempty node of $T^*(n)$ then both $P_a$ and $P_b$ contain the full complement of $j+m$ labels in $T^*(n)$; (2) if $P_b$ is the last nonempty node in $T^*(n)$, then $P_a$ contains $j+m$ labels and $P_b$ contains at most $j+m$ labels in $T^*(n)$; and (3) if $P_b$ is empty in $T^*(n)$ then $P_a$ contains at most $j+m$ labels in $T^*(n)$.

Consider the requirement in case (1). Here we observe that we can apply Lemma \ref{lem2} to both $P_a$ and $P_b$. Indeed, we can apply it to $P_a$ since $P_b$ will be full in $T(n)$. But we can also apply it to $P_b$ because there is a nonempty node $Q$ in $T^*(n)$ following $P_b$. $Q$ is not a leaf of $T^*(n)$ by definition of $P_b$. So $Q$ must be located above the penultimate level in $T(n)$. It will also contain $x$ labels in $T(n)$, for otherwise, it would become the last nonempty node in $T(n)$ and all its labels would be deleted during the end correction step. However, this would make $P_b$ the last nonempty node of $T^*(n)$. Therefore we have verified Lemma \ref{claim3} for $P_a$ and $P_b$ in case (1).

To verify cases (2) and (3) it suffices to show the following. Let $P$ be the last penultimate node of $T(n)$ that has at least one nonempty child in $T(n)$.
Then $P$ has between 0 to $j+m$ labels in $T^*(n)$. Indeed, if $P_b$ has a nonempty child in $T(n)$ (so $P = P_b$) then we can apply Lemma \ref{lem2} to $P_a$ in $T(n)$ due to $P_b$ being full. Hence, $P_a$ contains $j+m$ labels as a leaf of $T^*(n)$, and to establish (2) we need to show that $P_b$ contains between 0 to $j+m$ labels in $T^*(n)$. On the other hand, if both children of $P_b$ are empty in $T(n)$ (so $P = P_a$) then during the end correction step all labels from $P_b$ are removed. This makes $P_b$ empty in $T^*(n)$ and to show (3) we need to worry about the number of labels in $P_a$. Note that in both cases (2) and (3) node $P$ is the last non leaf node of $T(n)$ that is completely filled. So we have to show that if $T(n)$ is such that it contains a node $P$ which is its last penultimate node with a nonempty child and also its last filled non leaf node, then $P$ has between 0 to $j+m$ labels in $T^*(n)$.

Before we conclude the proof of Lemma \ref{claim3} we examine in a similar way what must be shown to prove Lemma \ref{claim4}. The first part of the lemma has already been covered in the above discussion, so what remains is to show that the number of nonempty cells of any leaf $P$ in $T^*(n)$ is equal to the number of nonempty cells of the left child of $P$ in $T(n)$.

Suppose the node $P$ is the last nonempty penultimate node of $T(n)$ and that its left child is empty in $T(n)$. Then as a result of the end correction step in the pruning process $P$ will be empty in $T^*(n)$, as required by Lemma \ref{claim4}. If $P$ is not the last nonempty penultimate node of $T(n)$ then
it is followed by a non leaf node $Q$ that is also nonempty in $T(n)$. Also, the left child of $P$ contains $j+m$ labels in $T(n)$. If $Q$ is completely filled in $T(n)$ then by Lemma \ref{lem2} the node $P$ will contain $j+m$ labels as a leaf of $T^*(n)$, and hence $j$ cells as required by Lemma \ref{claim4}. Thus, we are left to consider only one case: $P$ is the last penultimate node of $T(n)$ with a nonempty left child and where the first non leaf node $Q$ following $P$ contains less than $x$ labels in $T(n)$ ($Q$ could possibly be empty). We must show that in this case $P$ contains the same number of cells in $T*(n)$ as its left child does in $T(n)$.

Therefore, in order to complete the proofs of both Lemma \ref{claim3} and \ref{claim4} we must consider the following. Suppose the tree $T(n)$ contains a penultimate node $P$ with the property that it is the last penultimate node with a nonempty left child in $T(n)$ and that it is also the last completely filled non leaf node of $T(n)$. Then, we must show that as a leaf of $T^*(n)$ the node $P$ contains at most $j+m$ labels, and also that it has same number of nonempty cells as its left child does in $T(n)$. Until the end of this section let $P$ denote such a penultimate node, and denote by $L$ and $R$ the left and right child leaf of $P$ respectively ($R$ may be empty in $T(n)$). Our two requirements can be expressed as upper and lower bounds on the net number of labels being removed from and lifted into $P$ during the pruning operation.

In our proof we will establish these bounds on a case by case basis, where the cases (and subcases) are determined by the position of the label $n$ in $T(n)$. In doing so we will make frequent use of the following technical lemma:

\begin{lemma} \label{lem1}
Let $Y$ be a nonempty leaf in $T(n)$ and suppose that there are $\mu$ labels in $T(n)$ that
are situated at or after node $Y$ (so the smallest label in $Y$ is $n-\mu + 1$). If $\mu \leq j$ then
the number of labels removed from $Y$ during the pruning operation is $0$. If $\mu > j$ then
write $\mu - j = q(2j) + r$ with $0 \leq r < 2j$. In this case the number of labels removed from $Y$ and possibly its parent in $T(n)$ is $jq+ \min \{r,j\}$ provided that $q \leq p-1$; otherwise, the number of labels removed from $Y$ and possibly its parent in $T(n)$ is $pj$.
\end{lemma}

\begin{proof}
If $\mu \leq j$ then $Y$ is empty in all of the subtrees $T(n-j),\ldots, T(n-(2p-1)j)$, so the number of labels removed from $Y$ during the pruning operation is $0$.

If $\mu > j$ and $\mu - j = q(2j) + r$ with $q > p-1$ then all cells of $Y$ are nonempty in the aforementioned
$p$ subtrees. In this case we have already seen earlier that the number of labels removed from $Y$ and possibly its parent is $pj$.

Finally, suppose $\mu > j$ and $\mu - j = q(2j) + r$ with $0 \leq r < 2j$ and $q \leq p-1$.
Then the $q$ subtrees $T(n-j), \ldots, T(n-(2q-1)j)$ contain all
cells of $Y$  as being nonempty (if $q=0$ then none do). The subtree $T(n-(2q+1)j)$ contains
only the first $\min \{r,j\}$ cells of $Y$ as being nonempty, and for $i>q$ each of the remaining subtrees
$T(n-(2i+1)j)$ contains $Y$ as an empty leaf. Therefore, the net number of labels deleted
from $Y$ and its parent in $T(n)$ during the pruning operation is $jq+ \min \{r,j\}$.
\end{proof}

Now we proceed with the case by case analysis promised above.
\paragraph{Case 1} Suppose label $n$ is situated in node $L$.

\textbf{Subcase 1a:} Label $n$ is the $l^{th}$ label in $L$ with $1 \leq l \leq j$.
Thus $L$ is the last nonempty node in $T(n)$ and it contains $l \leq j$ labels.
The deletion step does not affect any labels in $L$ since $L$ is empty in all the subtrees $T(n-j), \ldots, T(n-(2p-1)j)$. Therefore, in the lifting step the $l$ labels from $L$ are inserted into $P$ and in the end correction step the $x$ largest labels are removed from $P$. This leaves $P$ with $l$ labels in $T^*(n)$, as required.

\textbf{Subcase 1b:} Label $n$ is the $(j + l)^{th}$ label in $L$ with $1 \leq l \leq m$ (if $m=0$ then this case is not needed).
By using Lemma \ref{lem1} the number of labels removed from $L$ and its parent is $qj + \min \{r,j\}$, where
$l = q(2j) + r$. Thus, in the lifting step of the pruning operation $l + j - qj-\min \{r,j\}$ labels are lifted into $P$, while in the end correction step the $x$ labels with the largest labels are removed from $P$. We need to verify that
\begin{equation} j \leq l + j - qj-\min \{r,j\} \leq m+j \end{equation} to establish Lemmas \ref{claim3} and \ref{claim4}.
The first inequality follows from $l = q(2j) + r \geq qj + \min \{r,j\}$ and the second follows from $m \geq l$.\\

\paragraph{Case 2:} Suppose label $n$ is situated in node $R$. To establish Lemmas \ref{claim3} and \ref{claim4}
we must show that $P$ contains at least $j$ labels and at most $m+j$ labels after pruning $T(n)$.

\textbf{Subcase 2a:} Label $n$ is the $l^{th}$ label in $R$ with $l \leq j$.
Here no labels are removed from $R$ during the deletion step.

If $l \geq (2p-1)j-m$ then $L$ loses $pj$ labels during the deletion step by Lemma \ref{lem1}.
In this case the number of labels in $P$ after pruning $T(n)$ is $m + l - (p-1)j$ and we need to establish that
\begin{equation}j \leq m + l -(p-1)j \leq m+j\,.\end{equation}
Both of these are clear because $l + m \geq (2p-1)j \geq pj$ and $l \leq j \leq pj$.

In the case $l < (2p-1)j-m$, by Lemma \ref{lem1} the number of labels removed from $L$ after the deletion step is
$qj + \min \{r,j\}$ with $m + l = q(2j) + r$. Then we need to verify that
\begin{equation}j \leq m + j  + l - qj - \min\{r,j\} \leq m+j\,.\end{equation}
The first inequality follows from $m+l = q(2j)+r \geq qj + \min\{r,j\}$. The second one follows because
$l \leq qj + \min\{r,j\}$ if either $q > 0$ or $r \geq j$ due to $l \leq j$. Otherwise, $m + l = r$ and so $l \leq r = \min \{r,j\}$,
which is the required inequality.

\textbf{Subcase 2b:} Label $n$ is the $(j + l)^{th}$ label in $R$ with $1 \leq l < (2p-2)j-m$ (if the rightmost term is
non positive then this case is not needed). In this case both children of $P$ may lose less than $pj$ labels.
To account for the number of labels lost by $L$ and $R$ during the deletion step, we write $m+j+l = q_L(2j) + r_L$
and $l = q_R(2j) + r_R$ with $0 \leq r_L, r_R < 2j$. Then by Lemma \ref{lem1}, the number of labels removed from nodes
$L$ and $R$ are  $q_L j + \min \{r_L,j\}$ and $q_R j + \min \{r_R,j\}$, respectively. Note that $q_R \leq q_L$. To prove
Lemmas \ref{claim3} and \ref{claim4} we need to show that
\begin{equation} \label{eqn2b} j \leq m+j + l + j - (q_L j + \min \{r_L,j\} + q_R j + \min \{r_R,j\}) \leq m+j \,.\end{equation}

For the first inequality we use the fact that $m+j+l = q_L(2j) + r_L$ to reduce the inequality to
$\min\{r_L,j\} + \min \{r_R,j\} \leq (q_L-q_R)j + r_L$. If $q_L > q_R$ then the latter inequality follows easily.
Otherwise, $q_L = q_R$ and so $l = q_R(2j) + r_R$ implies that $ m + j + r_R = r_L$. Thus, $r_L \geq j$
and the latter inequality becomes $j + \min \{r_R, j\} \leq r_L = m + j +r_R$; this is clearly true.

Now we consider the second inequality in (\ref{eqn2b}). Here we substitute $q_L j$ and $q_R j$ with
the equivalent values $\frac{m+l+j-r_L}{2}$ and $\frac{l - r_R}{2}$ respectively. Then after some
simplification the inequality becomes
\begin{equation} \label{eqn2b'} j + r_R + r_L \leq m + 2 \min \{r_L,j\} + 2\min \{r_R,j\}.\end{equation}
If both the minimums on the right are $j$ then we get $ r_R + r_L \leq m + 3j$, which is true due to
$r_L < 2j$ and $r_R \leq m+j$. If the minimums are $r_L$ and $r_R$ (both less than $j$) then the inequality
holds unless $m < j$. But when $m < j$ we have $q_R = 0$ and $q_L \in \{0,1\}$ because $l + j < 2j$ and
$m + l + 2j < 4j$. If $q_L = 0$ then
$r_L = m + l + j \geq j$, contradicting that $r_L < j$. Thus, $q_L = 1$ and the second inequality in
(\ref{eqn2b}) becomes $r_L \geq 0$.

Finally, if one of the minimums is $r_L$ or $r_R$ and the other is $j$ then the inequality in (\ref{eqn2b'})
is trivial.

\textbf{Subcase 2c:} Label $n$ is the $(j + l)^{th}$ label in $R$ with $l \geq (2p-2)j - m$. By the choice of $l$, this case treats
the situation where node $L$ loses $pj$ labels. By Lemma \ref{lem1} node $R$ loses $qj + \min \{r,j\}$ labels where
$l = q(2j) + r$. After pruning $T(n)$ the number of labels in $P$ is $$m-(p-1)j + l + j - qj-\min \{r,j\}\,.$$

To prove Lemma \ref{claim4} to need to establish that $j \leq m-(p-1)j + l + j - qj-\min \{r,j\}$.
Using the assumption $(2p-2)j - m \leq l$ and the fact that $l \leq m$, we deduce that $m \geq (p-1)j$.
Also, we have that $l = q(2j) + r \geq qj + \min \{r,j\}$. Together, these two observations imply that
$m+l \geq (p-1)j + qj + \min \{r,j\}$. This is the desired inequality above after simplification.

Now we verify the upper bound $m-(p-1)j + l + j - qj  - \min \{r,j\} \leq m+j$ that is required for Lemma \ref{claim3}.
This is equivalent to $l \leq (p-1)j + qj + \min \{r,j\}\,$. Using $l = q(2j) + r$ our upper bound is equivalent to
$$l + r \leq (2p-2)j + 2\min \{r,j\}\,.$$ If this inequality fails and $\min \{r,j\} = r$ then we have $l > (2p-2)j + r$.
But since $l \leq m \leq (2p-1)j$ this contradicts $l \equiv r \,(\mod 2j)$. On the other hand if the inequality fails with
$\min \{r,j\} = j < r$, then since $r < 2j$ we have that $(2p-2)j < l \leq (2p-1)j$. This contradicts $\min \{r,j\} = j < r$
since then the remainder $r$ is less than or equal to $j$.\\

\paragraph{Case 3:} Suppose label $n$ is located after node $R$ and is the $l^{th}$ label after the final label
in $R$ (thus, we have the understanding that $l \geq 1$). Since $P$ is the last non leaf node in $T(n)$ we have $l < x$.
Also, as $l < x$, the node $R$ does not lose a full set of $pj$ labels. To establish Lemmas \ref{claim3} and \ref{claim4}
we must show that $P$ contains between $j$ and $j+m$ labels after pruning.

\textbf{Subcase 3a:} $1 \leq l < 2((p-1)j-m) \leq x$ (note that for this to happen we need $m < (p-1)j$).
In this case both children of $P$ may lose less than $pj$ labels. Note that if this case does occur then
subcase 2c will not because that requires $m \geq (p-1)j$.

By Lemma \ref{lem1}, if $2m + l + j = q_L(2j) + r_L$ with $0 \leq r_L < 2j$ then $L$ loses
$q_L j + \min \{r_L,j\}$ labels. Also, if $m+l = q_R(2j) + r_R$ with $0 \leq r_R < 2j$ then R
loses $q_Rj + \min \{r_R,j\}$ labels.
After the end correction step we remove $x - l$ labels from $P$. So the number of labels in $P$
after pruning is $2(m+j) - (q_L j + \min \{r_L,j\} + q_Rj + \min \{r_R,j\}) + l$. We need to show that
\begin{equation} \label{eqn3a} j \leq 2(m+j) - (q_L j + \min \{r_L,j\} + q_Rj + \min \{r_R,j\}) + l \leq m+j\,.\end{equation}
The first inequality is $(q_L + q_R)j + \min \{r_L,j\} + \min \{r_R,j\} \leq 2m + j+ l = q_L(2j) + r_L$.
We are done if $q_L > q_R$. Otherwise, since $q_L \geq q_R$ we must have $q_L = q_R$ and this implies $r_L = m + j + r_R \geq j$.
Then the inequality reduces to $j + \min \{r_R,j\} \leq r_L = m + j + r_R$, which is true.

Now we consider the second inequality in (\ref{eqn3a}). It reduces to
$m+j+l \leq (q_L + q_R)j + \min \{r_L,j\} + \min \{r_R,j\}$, which in turn is the same as
$$(q_R + 1)j + r_R \leq q_Lj + \min \{r_L,j\} + \min \{r_R,j\}\,.$$
If $q_R < q_L-1$ then we are done by trivial considerations. Suppose that $q_R = q_L-1$.
We also get the inequality above easily if either $\min \{r_R,j\} = r_R$, or $\min \{r_R,j\} = j$ and
$\min \{r_L,j\} = j$. So we can assume that $\min \{r_R,j\} = j$ and $\min \{r_L,j\} = r_L$. Then using
$m + l = (q_L-1)(2j) + r_R$ and $2m + j + l = q_L(2j) + r_L$, we deduce that $m = j + r_L - r_R$.
However, as $m \geq 0$ we conclude that $r_R \leq j + r_L$ as required by the inequality above.

If $q_R = q_L$ then we need to show that $j + r_R \leq \min \{r_L,j\} + \min \{r_R,j\}$.
Once again we have $r_L = m + j + r_R$, and so $\min \{r_L,j\} = j$. The above then
becomes $r_R \leq \min \{r_R, j\}$, but we do have $r_R \leq j$ for otherwise $r_L > 2j$.

\textbf{Subcase 3b:} $2((p-1)j-m) \leq l < x$. In this case $L$ loses all $pj$ labels, and $R$,
by Lemma \ref{lem1}, loses $qj + \min \{r,j\}$ labels where $m + l = q(2j) + r$. The total number of
labels in $P$ after pruning is $(m - (p-1)j) + (m + j - qj - \min \{r,j\}) + l\,$. Hence we need to show
\begin{equation} \label{eqn3b} j \leq 2m+ j+ l - (p-1)j - qj - \min \{r,j\} \leq m+j\,.\end{equation}

We consider the first inequality of (\ref{eqn3b}). After substituting $m + l = q(2j) + r$ and simplifying,
the lower bound in (\ref{eqn3b}) becomes $(p-1-q)j + \min \{r,j\} \leq m+r$. If $q = p-1$ then the
latter inequality is obvious. Suppose that $q < p-1$.

We know that $l \geq (2p-2)j-2m$, which implies that $q(2j)+r = m+l \geq (2p-2)j-m$.
Thus, $m+r \geq 2(p-1-q)j$. Since $p-1-q \geq 1$, we conclude that
$m+r \geq (p-1-q)j + j \geq (p-1-q)j + \min \{r,j\}$ as needed.

Now consider the second inequality in (\ref{eqn3b}). Substituting $m + l = q(2j) + r$ we
get that $$qj + r \leq (p-1)j + \min \{r,j\}\,.$$ As $q \leq p-1$ we are done if $q < p-1$ or $r \leq j$.
If $q=p-1$ and $r > j$, then it follows that all $j$ cells in $R$ are nonempty in
$T(n-(2p-1)j)$; a contradiction since $m + l < (2p-1)j$.

With this we have considered all cases and the proofs of Lemma \ref{claim3} and Lemma \ref{claim4} are now complete.
We conclude this section by considering the frequency sequence $\phi_{C_T}$ of the cell counting function of a fixed tree $T = T_{s,j,m,p}$.

\subsection{The frequency sequence} \label{sec3:frequency}

From \cite{ConollyLike} we know that the tree-based solution sequence of the $(\alpha, \beta)$ Conolly
recursion (\ref{eq:alphabeta1}) has frequency sequence $\alpha + \beta \phi_C$ where $\phi_C$ is the
frequency sequence of the Conolly sequence (\ref{eq:0112}). This is the linear combination
$\frac{\alpha}{2} \phi_H + \beta \phi_C$ of the frequency sequences of the $H$ sequence (\ref{eq:0123}) and
the Conolly sequence. Now the function $C_T$ with the choice of simultaneous parameters
$(s,j,m,p) = (0,j, (\alpha+\beta-1)j, \alpha/2 + \beta)$ gives (\ref{rec2}), which is (\ref{eq:alphabeta1}) with the
simultaneous parameter $j$. So it is natural to wonder whether the frequency sequence of $C_T$ with
the aforementioned choice of parameters is $\frac{\alpha}{2} \phi_{H_j} + \beta \phi_{C_j}$ where $H_j$ and $C_j$
are the tree-based solutions of (\ref{eq:0j2j3j}) and (\ref{eq:0jj2j}) respectively for $s=0$.

Using the results about the frequency sequences of (\ref{eq:0j2j3j}) and (\ref{eq:0jj2j}) from Theorems 5.1 and 5.5
of \cite{Rpaper} we can easily compute $\frac{\alpha}{2} \phi_{H_j} + \beta \phi_{C_j}$. If $\nu_2(v)$ is the
2-adic valuation of $v$ then $\phi_C(v) = \nu_2(v) + 1$ and

\begin{displaymath}
\frac{\alpha}{2} \phi_{H_j}(v) + \beta \phi_{C_j} (v) = \left \{
\begin{array}{lr}
\frac{\alpha}{2} + \beta & \text{if}\; j \nmid v\\
\beta j \cdot \nu_2(\frac{v}{j}) + \frac{\alpha}{2}(j+1) + \beta & \text{otherwise} \end{array}
\right.
\end{displaymath}

On the other hand, we can derive $\phi_{C_T}$ using an argument most similar to that of Theorem \ref{thm:mfreq}. The difference is that
the non leaf regular nodes now contain $x$ labels instead of $j-m$. For fixed $(s,j,m,p)$

\begin{displaymath}
\phi_{C_T} (v) = \left \{
\begin{array}{lr}
1& \text{if}\; j \nmid v\\
((2p-1)j-m) \cdot \nu_2(\frac{v}{j}) + 1+m + s\mathbf{1}_{[\frac{v}{j} \; \text{is a power of 2}]}& \text{otherwise} \end{array}
\right.
\end{displaymath}

From this it is easy to see that $\phi_{C_T} \neq \frac{\alpha}{2} \phi_{H_j} + \beta \phi_{C_j}$ when
$(s,j,m,p) = (0,j,(\alpha+\beta-1)j,\alpha/2 + \beta)$. In the next section we derive a 2-ary order $p$
recursion whose solution sequence does indeed have the frequency sequence
$\frac{\alpha}{2} \phi_{H_j} + \beta \phi_{C_j}$ and we give a tree-based proof of the derivation.
\end{section}


\begin{section}{Linear combinations of frequency sequences via tree superpositions} \label{sec:linearcomb}

In this section we use the tree-based methodology to derive a nested recursion whose solution has a frequency sequence that is a linear combination of the frequency sequences of $H_{0,j}(n)$ and $R_{0,j}(n)$ from (\ref{eq:0j2j3j}) and (\ref{eq:0jj2j}) respectively for $s=0$. Our strategy is to construct a labelled infinite binary tree whose cell counting function has the desired property and then use the tree along with a pruning operation to derive a nested recursion
with the same frequency function.

To motivate our construction we recall the trees whose cell counting functions satisfy recursions $R_{0,j}(n)$ and $H_{0,j}(n)$. The first tree, say $T_1$ with
cell counting function $R_{0,j}(n)$, is $T_{0,j,0}$ from Section \ref{sec:order2}, that is, the binary tree corresponding to $m=0$. Similarly, the tree with cell counting function $H_{0,j}(n)$ is $T_2 = T_{0,j,j}$ from Section \ref{sec:order2}.
To obtain a tree $T$ whose cell counting function has a frequency sequence $\frac{\alpha}{2} \phi_{H_{0,j}} + \beta \phi_{R_{0,j}}$ we form the ``superposition" of the two trees $T_1$ and $T_2$. That is, we place $\alpha / 2$ copies of $T_2$ and $\beta$ copies of $T_1$ on top of each other. Note that since $T_1$ and $T_2$ have the same skeleton, this superposition creates another infinite binary tree $T$ with the same skeleton (see Figure \ref{fig:skeleton}). When we superpose multiple copies of $T_1$ and $T_2$ we initially treat the labels in each tree as placeholders; as a result, at first the labels in the superposed tree $T$ do not appear in preorder and $T$ has multiple occurrences of the same label (see Figure \ref{fig:superposition}, where for simplicity we illustrate the superposition of a single copy of each tree). Once we relabel the tree $T$ in preorder it is evident that we obtain a tree whose cell counting function has frequency sequence that is the desired linear combination $\frac{\alpha}{2} \phi_{H_{0,j}} + \beta \phi_{R_{0,j}}$ of the cell counting functions for the individual trees. Note that in principle $\alpha$ can be negative; in this case the tree $T$ is well-defined so long as $(\alpha /2) (j+1) + \beta \geq 1$ (that is, we require at least one label in the last cell in each leaf of $T$).

\begin{figure}[htpb]
\begin{center}
\includegraphics[scale=0.14]{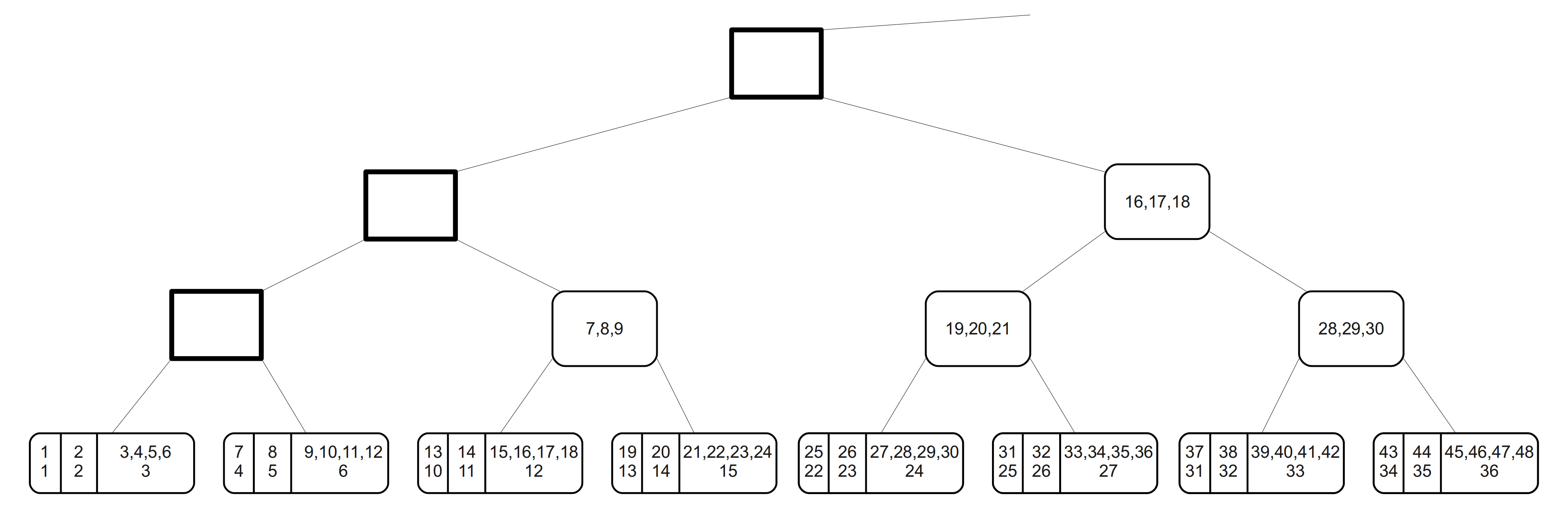}
\caption{Superposition of $T_{0,3,0}$ and $T_{0,3,3}$ prior to relabelling.} \label{fig:superposition}
\end{center}
\end{figure}

\subsection{The tree and the pruning operation} \label{sec4:pruning}

We now give a direct, more general construction of a tree $T$ whose cell counting function is, under certain conditions, the desired linear combination. This will allow us to solve not only a nested recursion whose solution has the desired frequency function, but also a wide spectrum of related recursions.

Fix simultaneous parameters $s,j,m,p$. The natural range of these parameters is discussed below. The desired tree $T = T_{s,j,m,p}$ has the skeleton of the infinite binary tree from Figure (\ref{fig:skeleton}) with $j$ cells in each leaf. For $n \geq 1$ let $T(n)$ denote $T$ with $n$ labels inserted in preorder as follows: each of the first $j-1$ cells of each leaf receives $p$ labels, while the last cell receives $p + m$ labels. All remaining regular nodes in $T$ get $pj - m$ labels each, and the supernodes receive $s$ labels each. See Figure \ref{fig:T87}, where we use the case $s = 0, j=m=3, p=2$ as our running example.

\begin{figure}[htpb]
\begin{center}
\includegraphics[scale=0.14]{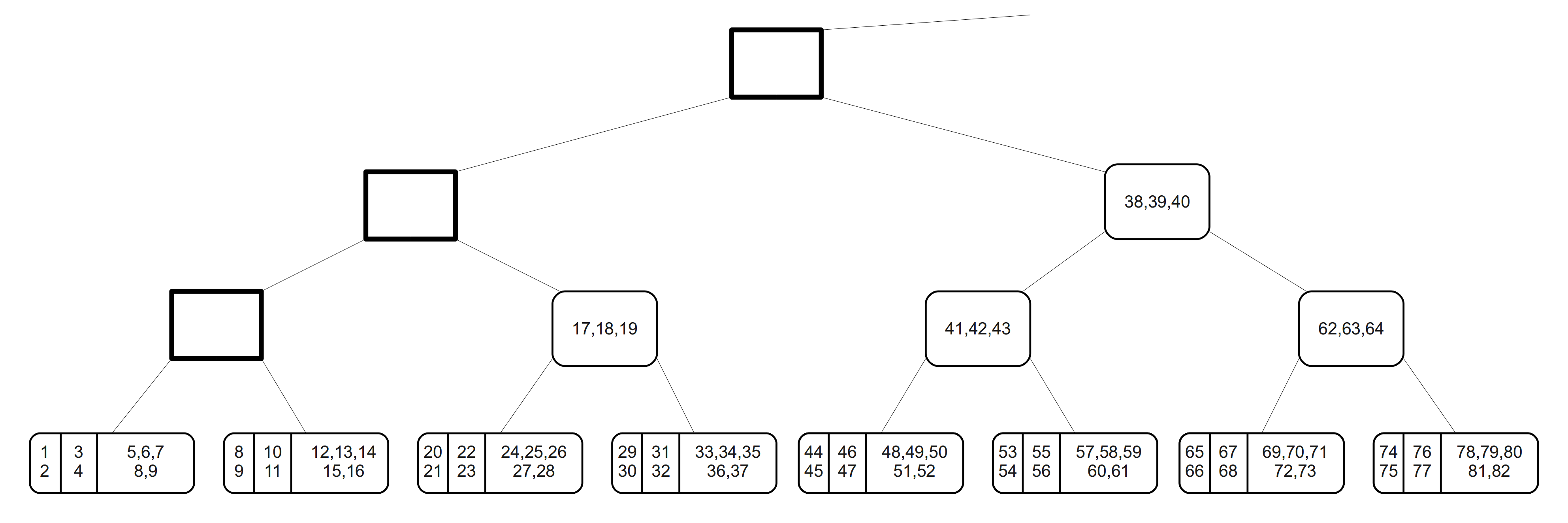}
\caption{The labelled infinite tree $T_{0,3,3,2}(82)$; $C_T(82)=24$.} \label{fig:T87}
\end{center}
\end{figure}

To ensure that each leaf has at least one cell and that cells have a positive number of labels, we require $p,j \geq 1$. Likewise, to force regular nodes and supernodes to contain a non-negative number of labels, we need $s \geq 0$ and $0 \leq m \leq pj$. Note that for negative values of $m$ such that $m > -p$, the tree $T$ is still well-defined. However, our proof here only works for non-negative $m$, so we restrict the range of this parameter accordingly. See Section \ref{sec:conclusion} for further discussion.

We let $C_T(n)$ denote the number of nonempty cells in $T(n)$. Since each cell has $p$ labels and each regular nodes contains $pj-m$ labels, the same argument as in the proof of Theorem \ref{thm:mfreq} yields that
\begin{displaymath}
\phi_{C_T} (v) = \left \{
\begin{array}{lr}
p& \text{if}\; j \nmid v\\
(pj-m) \cdot \nu_2(\frac{v}{j}) + p+m + s\mathbf{1}_{[\frac{v}{j} \; \text{is a power of 2}]}& \text{otherwise} \end{array}
\right.
\end{displaymath}

Observe that when $s = 0$ and $m = bj$ for some integer $b \geq 0$, the frequency sequence $\phi_{C_T} = b \phi_{H_{0,j}} + (p-b) \phi_{R_{0,j}}$. That is, for $m$ a multiple of $j$, the resulting tree $T$ is a superposition of trees $T_1$ and $T_2$ as discussed above. Note that the restriction $m \geq 0$ allows us to only produce the frequency functions which are linear combinations of $\phi_{H_{0,j}}$ and $\phi_{R_{0,j}}$ with non-negative coefficients. If $m$ is not a multiple of $j$ then the resulting tree is not a superposition of trees $T_1$ and $T_2$.

Let $\alpha, \beta \geq 0$, and set $j = 1$, $p = \alpha / 2 + \beta$ and $m = \alpha / 2$. Then $\phi_{C_T}$ is exactly $\frac{\alpha}{2} \phi_{H} + \beta \phi_{C}$, that is, $\phi_{C_T}$ is the frequency function of the solution for the recursion (\ref{eq:alphabeta1}). We now generalize this result by deriving a recursion whose solution has a frequency function that is a linear combination of the frequency functions for the solutions to $R_{0,j}$ and $H_{0,j}$.

In Section \ref{sec:order2} and \ref{sec:orderp} we have used pruning operations to show that a cell counting function is the solution to a nested recursion. Here we reverse our approach and use a pruning operation to derive a recursion whose solution is given by $C_T(n)$.

The two major requirements we place on the pruning operation is that the resulting tree $T^*(n)$ has the same skeleton as $T$ and that it conforms to the labelling rules described earlier. In particular, we would like our pruning operation to be defined in such a way that a ``typical" nonempty leaf of $T(n)$ loses $pj$ labels in the deletion step. In that case, $2m$ labels are lifted to its parent to bring the total count of labels in it to $pj+m$ after pruning (exactly the number of labels in a ``typical" nonempty leaf). Using this heuristic, and examples of pruning operations in Section \ref{sec:order2} and Section \ref{sec:orderp}, we define the following pruning operation on $T(n)$, $n > 5pj + 3m + 2s$ (the first seven nodes of $T$ are full). See Figures \ref{fig:s=0_p=2_j=m=3_initialcorrection}, \ref{fig:s=0_p=2_j=m=3_deletion}, \ref{fig:s=0_p=2_j=m=3_lifting}, \ref{fig:s=0_p=2_j=m=3_endcorrection} and \ref{fig:s=0_p=2_j=m=3_afterpruning}.

\begin{figure}[h]
 	\includegraphics[scale=.14]{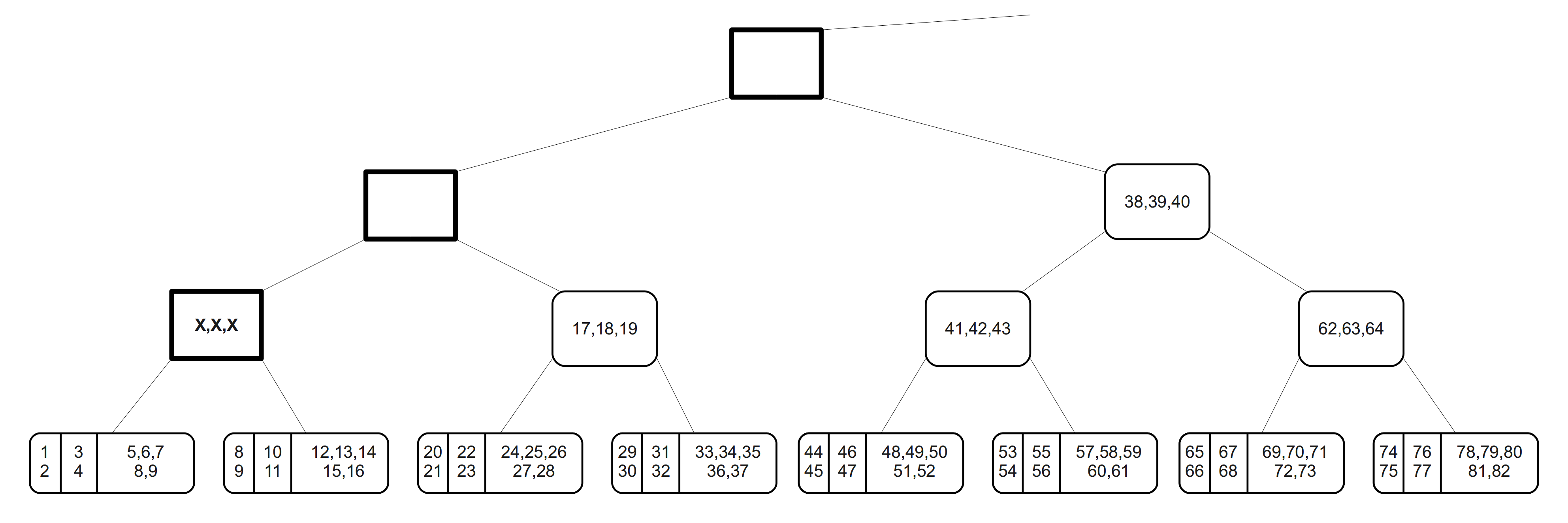}
 \caption{Initial correction step when pruning $T_{0,3,3,2}(82)$.}\label{fig:s=0_p=2_j=m=3_initialcorrection}
\end{figure}

\begin{figure}[h]
	\includegraphics[scale=.14]{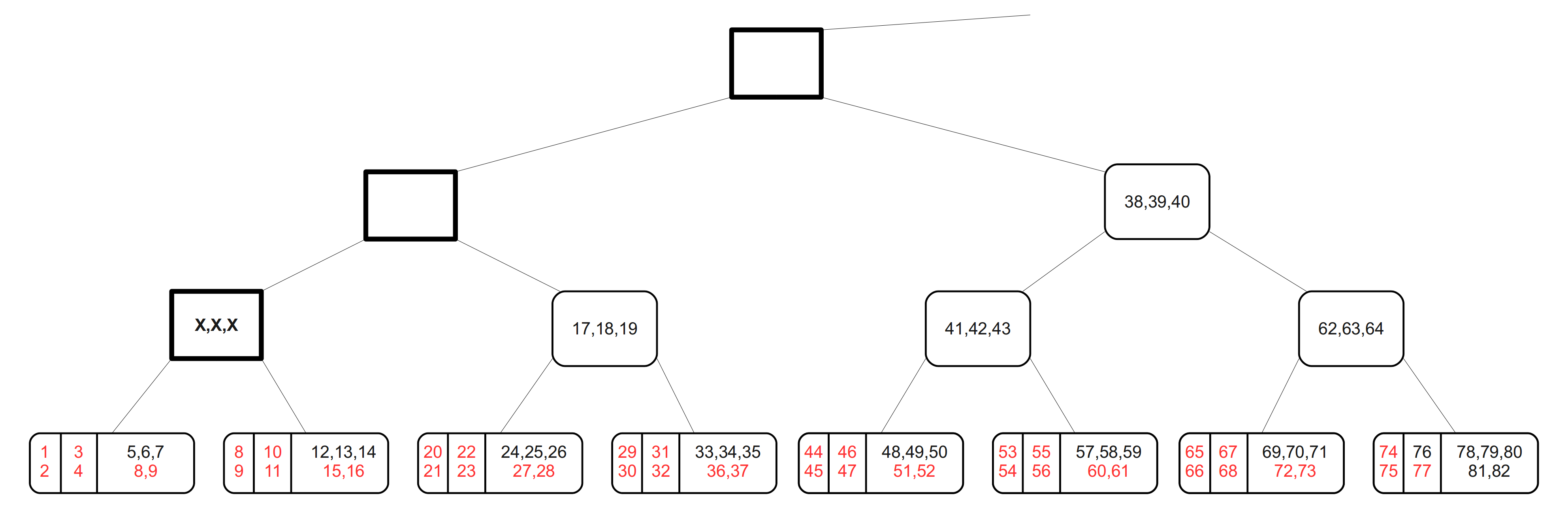}
\caption{Deletion step when pruning $T_{0,3,3,2}(82)$.}\label{fig:s=0_p=2_j=m=3_deletion}
\end{figure}

\begin{figure}[h]
	\includegraphics[scale=.14]{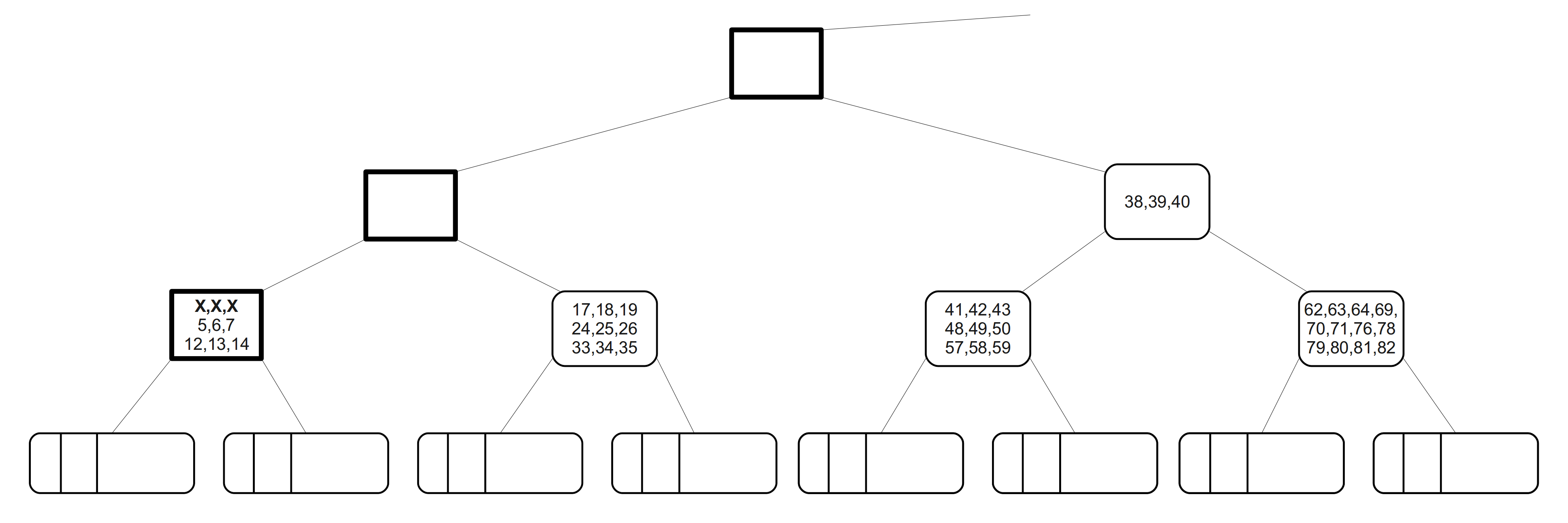}
  \caption{Lifting step when pruning $T_{0,3,3,2}(82)$.}\label{fig:s=0_p=2_j=m=3_lifting}
\end{figure}

\begin{figure}[h]
	\includegraphics[scale=.14]{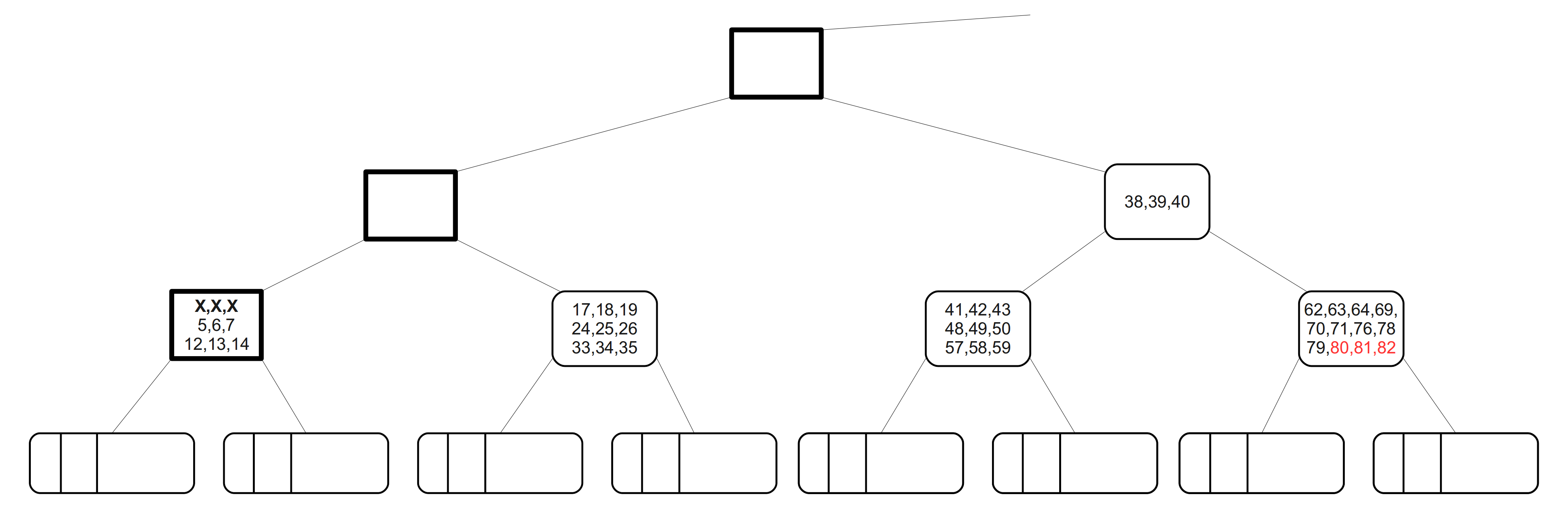}
  \caption{End correction when pruning $T_{0,3,3,2}(82)$.}\label{fig:s=0_p=2_j=m=3_endcorrection}
\end{figure}

\begin{figure}[h]
	\includegraphics[scale=.14]{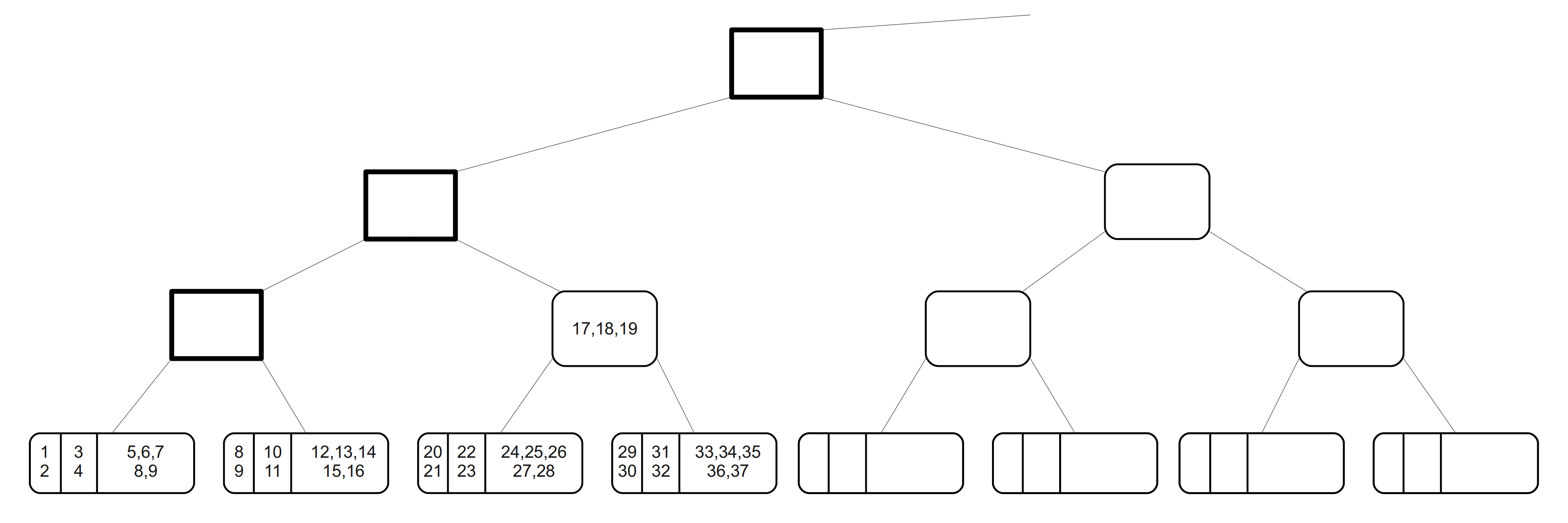}
  \caption{Relabelling step when pruning $T_{0,3,3,2}(82)$.}\label{fig:s=0_p=2_j=m=3_afterpruning}
\end{figure}

\begin{description}
\item[Initial correction] Remove the $s$ labels from the first supernode and insert $pj-m$ labels in the first supernode (these labels are currently placeholders only; we do not relabel the tree until the relabelling step).
\item[Deletion step] For each $i$, $1 \leq i \leq p$, consider the subtrees $T(n-(2i-1)-p(j-1))$. For each nonempty cell in the subtree $T(n-(2i-1)-p(j-1))$, delete a label from the corresponding cell in $T(n)$ (to be specific, we remove the largest label in the cell). Note that this is always possible because every nonempty cell has at least $p$ labels.
\item[Lifting step] Lift all the remaining labels from every nonempty cell of $T(n)$ into the parent of the leaf containing that cell. After this step all bottom level leaves of $T(n)$ become empty.
\item[End correction] Remove the largest $pj-m$ labels that remain in $T(n)$.
\item[Relabelling step] Remove all the leaves of $T(n)$ and relabel the new tree in preorder. Partition the new leaf labels into $j$ cells in the same manner as was done for the leaves of $T$ and denote this tree by $T^*(n)$.
\end{description}

The number of labels removed from $T(n)$ after the pruning operation is $s+\sum_{i=1}^p C_T(n-2i+1-p(j-1))$.
Thus, $T^*(n)$ contains $n-s-\sum_{i=1}^p C_T(n-2i+1-p(j-1))$ labels and it has the same skeleton as $T$, but it is not immediately obvious that $T^*(n)$ follows the labelling scheme defined earlier. In other words, we would like to establish the following lemma:

\begin{lemma}
The pruning of $T(n)$ results in a tree $T^*(n) = T(n-s-\sum_{i=1}^p C_T(n-(2i-1)-p(j-1)))$.
\label{lm:prunlincomb1}
\end{lemma}

To derive a recursion satisfied by $C_T(n)$ we will also need to establish a bijective correspondence between the cells of $T^*(n)$ and the cells of left leaves of $T(n)$:

\begin{lemma}
Let $P$ be a nonempty leaf in $T^*(n)$ (so that $P$ is a penultimate level node in $T(n)$). Then the number of nonempty cells of $T^*(n)$ in $P$ is equal to the number of nonempty cells of the left child of $P$ in $T(n)$.
\label{lm:prunlincomb2}
\end{lemma}

\subsection{The main theorem} \label{sec4:main result}

Once we establish Lemmas \ref{lm:prunlincomb1} and \ref{lm:prunlincomb2} we can derive a nested recursion as follows. Let $C_{T,L}(n)$ and $C_{T,R}(n)$ be the number of nonempty cells in $T(n)$ that are located in the left and right leaves respectively. Thus $C_T(n) = C_{T,L}(n) + C_{T,R}(n)$. Since there is a bijection between nonempty cells in the right leaves of $T(n)$  and nonempty cells in the left leaves of $T(n - pj -m)$, we have that $C_{T,R}(n) = C_{T,L}(n-pj-m)$. Therefore, by Lemma \ref{lm:prunlincomb1} and \ref{lm:prunlincomb2}
\begin{eqnarray*}
C_{T}(n) &=& C_T\left( n - s - \sum_{i=1}^p C_T(n-(2i - 1) - p(j-1))\right)\\
         &+& C_T\left(n-s-pj-m-\sum_{i=1}^p C_T(n-(2i-1)-m-p(2j-1))\right)
\end{eqnarray*}
for $n > 5pj + 3m + 2s$. That is, Lemma \ref{lm:prunlincomb1} and \ref{lm:prunlincomb2} together establish the following result:

\begin{theorem}
The cell counting function $C_T(n)$ satisfies the 2-term order $p$ nested recursion:

\begin{eqnarray}
R(n) &=& R \left( n - s - \sum_{i=1}^p R(n-(2i - 1) - p(j-1))\right) \label{eqn3} \\
&+& R \left(n-s-pj-m-\sum_{i=1}^p R(n-(2i-1)-m-p(2j-1)\right)\,. \notag{} \end{eqnarray}

In particular, recursion $(\ref{eqn3})$ generates the cell counting function $C_T(n)$ if it is given $5pj + 3m + 2s$
initial conditions (every node until the right child of the first regular node is full) that agree with the cell counting function.
\label{thm:lincomb}
\end{theorem}

\subsection{Proof of Theorem \ref{thm:lincomb}} \label{sec4:proof}

We now proceed with the proof of Lemmas \ref{lm:prunlincomb1} and \ref{lm:prunlincomb2}. The trees $T^*(n)$ and $T(n-s-\sum_{i=1}^pR(n-2i+1-p(j-1))$ contain the same number of labels and have the same skeleton. Therefore, to establish the desired results we need to show that the nodes of $T^*(n)$ contain the ``correct" number of labels.

As in Sections \ref{sec:order2} and \ref{sec:orderp}, we can think of every node in $T^*(n)$ as being part of $T(n)$. By the design of the pruning operation on $T(n)$, all nodes in $T^*(n)$, except for the leaves of $T^*(n)$ (which are the former penultimate nodes of $T(n)$) and the last nonempty node (which may or may not be a leaf of $T^*(n)$), contain the same number of labels as they do in $T(n)$. So to prove Lemma \ref{lm:prunlincomb1} we need only focus on the leaves of $T^*(n)$ and its last nonempty node.

Consider first the last nonempty node of $T^*(n)$. Suppose that it is \emph{not} a leaf of $T^*(n)$. Then in $T(n)$ this node is neither a leaf nor a penultimate node. Therefore the pruning operation on $T(n)$ doesn't add any labels to this node (the end correction step of the pruning may remove some labels). After pruning, this last nonempty node in $T^*(n)$ has at most the same number of labels that it has in $T(n)$, which is what we require.

To count the number of labels that remain in leaves of $T^*(n)$, we have the following lemma.

\begin{lemma}
Let $P$ be a nonempty penultimate level node in $T(n)$ and let $l_P$ be the number of labels in nodes of $T(n)$ after $P$ (in preorder).
\begin{enumerate}
\item If $l_P = 0$ then $P$ is empty in $T^*(n)$.
\item If $1 \leq l_P \leq 3pj + m$ then $P$ contains
\begin{align*}
l_P -& \sum_{i=1}^p \min\left(j, \left\lceil \frac{l_P-p(j-1)-2i+1}{p} \right\rceil \cdot \mathbf{1}_{[l_P-p(j-1)-2i+1 > 0]} \right)\\
  -& \sum_{i=1}^p \min\left(j, \left\lceil \frac{l_P -p(2j-1)-m-2i+1}{p} \right\rceil \cdot \mathbf{1}_{[l_P -p(2j-1)-m-2i+1 > 0]} \right)
\end{align*}
labels in $T^*(n)$.
\item If $l_P > 3pj + m$ then $P$ contains $pj+m$ labels in $T^*(n)$.
\end{enumerate}
\label{lm:limcombcount}
\end{lemma}

To simplify the notation, let $h(k, x) = k - p(j-1) + 1 - x$ and
\begin{align*}
d(l) = \sum_{i=1}^p \min\left(j, \left\lceil \frac{h(l, 2i)}{p} \right\rceil \cdot \mathbf{1}_{[h(l, 2i) > 0]} \right)
\end{align*}
Also, when there is no confusion we write $l$ instead of $l_P$. Note that the expression in (2) above reduces to $l-d(l)-d(l-pj-m)$. Further, the deletion step of the pruning operation can now be rephrased in terms of the subtrees $T(h(n, 2)), T(h(n, 4)), \ldots, T(h(n, 2p))$.

\begin{proof}
To prove (1), we note that if all the nodes of $T(n)$ after $P$ are empty and $P$ is not, then $P$ contains the largest label in $T(n)$. During the end correction step of the pruning the $pj - m$ largest labels are removed from the tree. Since $P$ contains at most $pj - m$ labels, it is emptied by the pruning operation.

Next we prove (3). Let $a = n-l$. Then
\begin{align*}
n - p(j-1) - (2p-1) &= (l - p(j-1) - (2p-1)) + a = h(l, 2p) + a\\
                            & > h(3pj + m, 2p) + a = pj + m + p(j-1) + 1 + a
\end{align*}
since $l > 3pj + m$. Thus the label $n - p(j-1) - (2p-1)$ is no further back in the tree then the last label of the right child of $P$. In other words, every cell of left and right child of $P$ is nonempty in each of
\begin{align*}
T(h(n, 2)), T(h(n, 4)), \ldots, T(h(n, 2p))
\end{align*}
Therefore, each cell in the children of $P$ will lose exactly $p$ labels and $2m$ labels will be lifted to $P$ on the lifting step. Note that none of the labels are removed from $P$ in the end correction step of the pruning operation since there are at least $pj-m$ labels in nodes of $T(n)$ after the children of $P$. Thus, $P$ has $pj - m + 2m$ labels after pruning.

Now we prove (2). Let $L$ and $R$ be the left and right child of $P$ respectively. Recall that we assume here that $1 \leq l \leq 3pj + m$. Note that if $h(l, 2i) \leq 0$ for some $i$ with $1 \leq i \leq p$, then $L$ has no nonempty cells in $T(h(n, 2i)$ and thus no labels in $L$ are pruned when we consider $T(h(n, 2i)$. On the other hand, if $0 < h(l,2i) \leq pj$, then $T(h(l,2i)$ will have $\lceil \frac{h(l,2i)}{p} \rceil$ nonempty cells in $L$ and which is exactly the number of labels removed from $L$ when considering $T(h(n, 2i))$. Similarly, if $h(l,2i) > pj$, then all $j$ cells of $L$ are nonempty in $T(h(n,2i)$ and $j$ labels are removed from $L$ when considering this subtree. Therefore, it follows that $d(l)$ is the number of labels that are removed from $L$ during pruning. Since there are $l-pj-m$ labels in nodes of $T(n)$ after $R$, we may repeat this argument to obtain that $d(l-pj-m)$ is the number of labels removed from $R$ during pruning.

Let $l = l_1 + l_2$ where $l_1$ is the total number of labels in $L$ and $R$  before pruning, and $l_2 \leq pj - m$ is the number of labels in nodes of $T(n)$ after $R$. Then there will be $p j - m + l_1 - d(l) - d(l-pj-m)$ labels on $P$ before the end correction step. During the end correction step we remove $p j -m$ largest labels from the tree. Namely, $l_2$ labels will be removed from the nodes that follow $R$ in preorder and the remaining $l_2 - p j +m$ will be taken from $P$, leaving exactly $l - d(l) - d(l-pj-m)$ labels in it.
\end{proof}

If $P$ is not one of the two last nonempty penultimate nodes in $T(n)$, then there are at least $3pj+m$ labels in the nodes that follow it. Thus, from the lemma, after pruning $P$ will contain $pj + m$ labels. Now, if $P$ and $Q$ are the last and second last nonempty penultimate nodes, respectively, and all children of $P$ are empty, then $P$ will be empty in $T^*(n)$.  If $l_Q > 3pj + m$ then $Q$ contains $pj + m$ labels; otherwise there are only $l_Q - d(l_Q) - d(l_Q - pj-m)$ labels in $Q$. If $P$ has a nonempty child then $l_Q > 3pj+m$ and $Q$ will have $pj+m$ labels in $T^*(n)$. If $l_P > 3pj + m$ then $P$ contains $pj + m$ labels; otherwise there are only $l_P - d(l_P) - d(l_P - pj-m)$ labels in $P$. Therefore, to prove Lemma \ref{lm:prunlincomb1} it remains to verify that if $P$ is a penultimate node and $l_P \leq 3pj+m$, then $l_P - d(l_P) - d(l_P - pj - m) \leq pj + m$.

Furthermore, observe that if $l_P \leq p(j-1) + 1$ then $d(l_P) = d(l_P - pj - m) = 0$, i.e. the number of labels in $P$ after pruning is $l_P$ and in that case both Lemma (\ref{lm:prunlincomb1}) and Lemma (\ref{lm:prunlincomb2}) hold. Therefore, to complete the proof of these lemmas it suffices to check that the following result holds.
\begin{lemma}
For $p(j-1) + 1 < l \leq 3pj+m$, $p(j-1) + 1 \leq f(l) \leq pj + m$ where $f(l) = l - d(l) - d(l - pj -m)$.
\label{lm:fbounds}
\end{lemma}

First we restrict our attention to $d(l)$. The following Lemma completely determines behaviour of $d(l)$ for the specified range of $l$.

\begin{lemma}
For $p(j-1)+ 2 \leq l \leq 3pj + m$, $d(l)$ is a non-decreasing function. In particular, as $l$ changes from $p(j-1)+2$ to $p(j-1) + p$ the function $d$ grows from $d(p(j-1)+2) = 1$ to $d(p(j-1)+p) = \lfloor \frac{p}{2} \rfloor$. For $2pj \geq l > p(j-1) + p$, if $p$ is odd then $d(l+1) -d(l) = 1$ and if $p$ is even then $d(l+1) - d(l)$ alternates between 0 and 2 if $p$ is even. For $l > 2pj$, $d(l) = pj$.
\label{lm:dbounds}
\end{lemma}
\begin{proof}
 Note that when $l$ increases by 1, each summand in $d(l)$ either increases by 1 or stays the same. It follows, that $d(l)$ is a non-decreasing function. To prove the rest of the Lemma we need to understand how many summands in $d(l)$ can increase at the same time. That is, we need find how many of the $h(l,2), h(l,4), \ldots h(l,2p)$ can be multiples of $p$ at the same time.

Consider two intervals $S_1 = [h(l,1), h(l,p)]$ and $S_2 = [h(l,p+1), h(l,2p)]$. The integers in $S_1 \cup S_2$ of the form $h(l,2i)$ correspond to the summands in $d(l)$. Also, note that each of these intervals contain exactly one multiple of $p$.

If $p$ is odd and $h(l,x)$ and $h(l,y)$ are multiples of $p$ from the first and second list respectively then it follows that $y = x + p$ and hence one of $x,y$ is even and the other one is odd. We also note that if $l$ is increased by 1 then the roles of $x$ and $y$ are interchanged, i.e. if $x$ was even and $y$ was odd, then after $l$ is increased $x$ is odd and $y$ is even. Therefore, if $p$ is odd there is always exactly one multiple of $p$ among $h(l,2), h(l,4), \ldots h(l,2p)$, i.e. each time $l$ is increased by 1 exactly one of $\lceil \frac{h(l,2)}{p} \rceil, \ldots, \lceil \frac{h(l,2p)}{p} \rceil$ increases by 1 as well. We also note that the increasing terms alternate between the $S_i$: if an increase in $l$ by 1 leads to an increase in $\lceil \frac{h(l,2u_1)}{p} \rceil$ and $h(l,2u_1)$ belongs to $S_1$ then increasing $l$ again leads to an increase in $\lceil \frac{h(l,2u_2)}{p} \rceil$ with some $h(l,2u_2)$ in $S_2$. The analysis in this paragraph is also valid for $p=1$.

Similarly, if $p$ is even then $x$ and $y$ have the same parity. Moreover, if $x,y$ are even then once $l$ is increased by 1, they both become odd and vice versa. Therefore, if $p$ is even either none or exactly two of $\lceil \frac{h(l,2)}{p} \rceil, \ldots, \lceil \frac{h(l,2p)}{p} \rceil$ grow by 1 when $l$ increases by 1. Thus, the difference sequence $d(l+1)-d(l)$ alternates between 0 and 2.

Finally, we are ready to fully describe the behaviour of $d(l)$. We observe that when $p(j-1)+1 < l \leq p(j-1) + p$ the function $d$ grows from $d(p(j-1)+2) = 1$ to $d(p(j-1)+p) = \lfloor \frac{p}{2} \rfloor$. This is because each summand of $d(l)$ corresponding to indices in $S_1$ increases by 1 and each summand corresponding to indices in $S_2$ remains zero because the indicator function $\mathbf{1}_{[h(l, 2i) > 0]}$ will be zero for those summands. For $2pj \geq l > p(j-1) + p$, $d(l)$ either satisfies $d(l+1) -d(l) = 1$ if $p$ is odd, or $d(l+1) - d(l)$ alternates between 0 and 2 if $p$ is even. As we have noted earlier, $d(l) = pj$ for $l > 2pj$.
\end{proof}

It follows from the Lemma \ref{lm:dbounds} that a similar result holds for $d(l-pj - m)$. As $l$ increases from $pj+m + p(j-1)+2$ to $pj + m +  p(j-1) + p$, the function $d$ grows from $1$ to $\lfloor \frac{p}{2} \rfloor$. After that $d(l -pj -m)$ is either a slowly growing sequence or has successive differences that alternate between 0 and 2. Once $d$ reaches $pj$, it remains constant.

Now we are ready to prove Lemma \ref{lm:fbounds}.

\begin{proof}
Recall that we would like to establish that for $p(j-1) + 1 < l \leq 3pj+m$, $p(j-1) + 1 \leq f(l) \leq pj + m$. It follows from Lemma \ref{lm:dbounds} that $f(l)$ grows from $p(j-1)+1$ at $l = p(j-1)+2$ to $p(j-1)+ \lceil \frac{p}{2} \rceil$ at $l = p(j-1) + p$. For $p(j-1) + p \leq l \min(pj + m + p(j-1) + 1, 2pj)$, $f(l)$ remains constant if $p$ is odd or alternates between $p(j-1) + \lceil \frac{p}{2} \rceil + 1$ and $p(j-1) + \lceil \frac{p}{2} \rceil$ if $p$ is even.  Thus, $f(l)$ lies within the required bounds for $p(j-1) + 1 < l \leq \min(pj + m + p(j-1) + 1, 2pj)$.

Now we consider two cases: $2pj \leq pj + m + p(j-1)+1$ and $pj + m + p(j-1)+1 < 2pj$. If $2pj \leq pj + m + p(j-1)+1$ then $f(l)$ is a increasing for $l$ in $[2pj, pj + m + p(j-1) + 1]$ with $f(pj + m + p(j-1) + 1) = p(j-1) + 1 +m < pj + m$ since $d(l) = pj$ for $l \geq 2pj$. Also, since $f(l)$ is increasing in this case, we still have $f(l) \geq p(j-1) + 1$ for $l$ in the given range.

If $pj + m + p(j-1)+1 < 2pj$ then both $d(l)$ and $d(l-pj-m)$ grow at the same time and $f(l)$ can potentially fall below $p(j-1) + 1$. However, note that $d(l-pj-m)$ grows only up to $\lfloor \frac{p}{2} \rfloor$ between $pj +m+ p(j-1)+1$ and $pj + m + p(j-1)+p = 2pj + m$ by Lemma \ref{lm:dbounds}. Therefore,
$f(l)$ can only decrease to $p(j-1) + 1$ on this interval. From the previous case it follows that the upper bound $f(l) \leq pj + m$ still holds in this case as well.

Therefore, $f(l)$ lies within the required bounds for $\min(pj + m + p(j-1) + 1, 2pj) < l \leq \max(pj + m + p(j-1) + 1, 2pj)$.

For $\max(pj + m + p(j-1) + 1, 2pj) \leq l  \leq p(j-1) + p + pj + m$, $f(l)$ is a non-decreasing function with $f(p(j-1) + p + pj + m) = p(j-1) + m + \lceil \frac{p}{2} \rceil \leq pj + m$. Thus, for $l$ in this range $f(l)$ also lies within the required bounds.

Finally, for $l$ in $[p(j-1) + p + pj + m, 3pj + m]$, $f(l)$ is either $p(j-1) + m + \lceil \frac{p}{2} \rceil$  ($p$ is odd) or  alternates between $p(j-1) + m + \lceil \frac{p}{2} \rceil$ and $p(j-1) + m + \lceil \frac{p}{2} \rceil + 1$ ($p$ is even). Therefore, $f(l)$ is within the required range for these values of $l$ as well and the proof is complete.
\end{proof}

\end{section}


\begin{section}{Nested $k$-ary order $p$ recursions} \label{sec:karyorderp}

In this section we continue our study of nested recursions via the lens of simultaneous parameters by extending our earlier approach to solve certain $k$-ary, order $p$ recursion families. We begin our discussion by reviewing previous work on $k$-ary nested recursions of type (\ref{eq:Conolly}).

The $k$-ary Conolly recursion
\begin{equation} \label{ktermConolly}
C_k(n) = \sum_{i=1}^k C_k(n - i + 1 - C_k(n-i)) \,. \end{equation}
is studied in \cite{DR}. There it is shown that the solution to (\ref{ktermConolly}), with appropriate initial conditions, counts leaves on the infinite, labelled, $k$-ary tree which is the natural extension of the infinite binary tree associated with the solution to the usual Conolly recursion (\ref{eq:0112}). Further, it is shown that the frequency sequence of this solution is $\phi_{C_k}(v) = \nu_k(v) + 1$, where $\nu_k(v)$ is the $k$-adic valuation of $v$.

As it will be required in what follows, we describe the infinite, labelled $k$-ary tree used above, which reduces to the binary tree we described earlier when $k=2$. There are supernodes along the leftmost spine and regular nodes. The first supernode has $k$ leaf children; every other supernode has $k-1$ regular nodes plus a supernode for their $k$ children.  Apart from the leaves, all regular nodes also have $k$ children. See Figure \ref{fig:3aryskeleton} for the skeleton of the tree for $k=3$. The nodes are labelled in preorder, with $s$ labels in each supernode and one label in each regular node. For (\ref{ktermConolly}), $s=0$ and $C_k(n)$ counts the number of leaves up to the $n^{th}$ label.

\begin{figure}
\includegraphics[scale=.14]{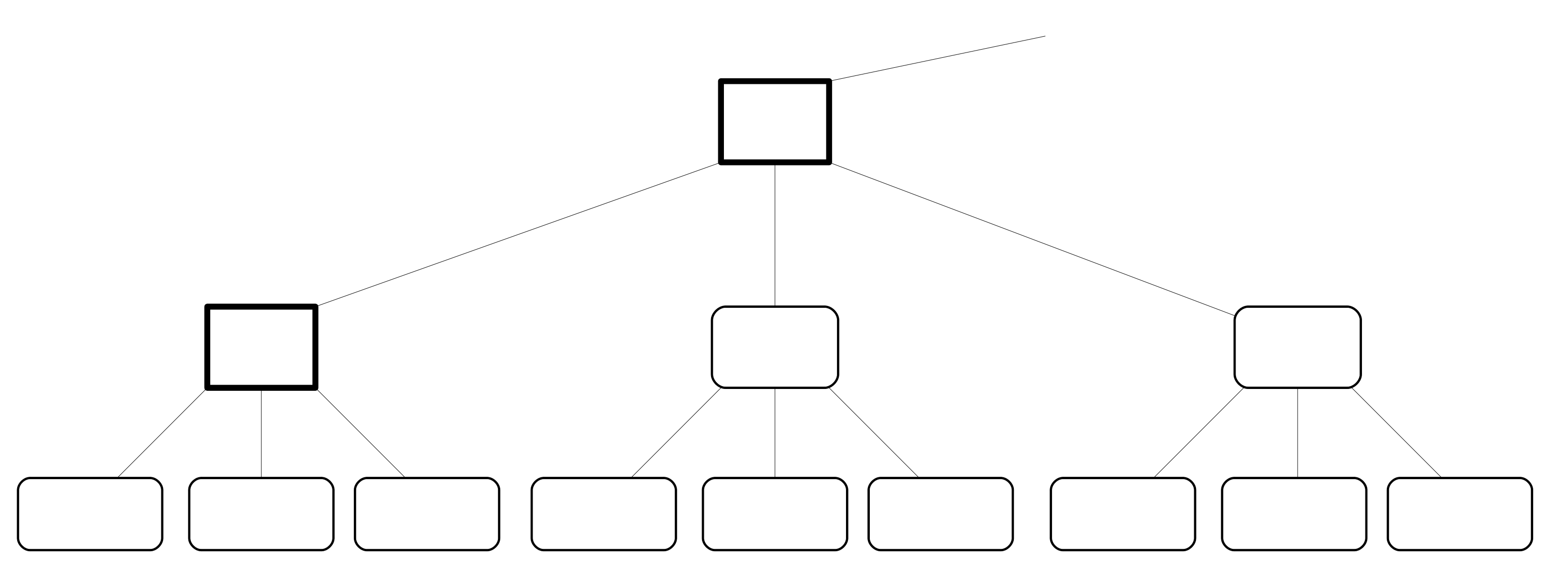}
 \caption{The skeleton of the infinite $k$-ary tree for $k=3$.}\label{fig:3aryskeleton}
\end{figure}

Recall that the ceiling function $\lceil \frac{n}{2} \rceil$ is the solution to the $H$ recursion (\ref{eq:0123}) (see, for example, \cite{BLT, ConollyLike}. This result is generalized in \cite{Isgur}, where it is shown that $\lceil \frac{n}{k} \rceil$ is the solution to the following $k$-ary, order $k-1$ recursion with appropriate initial conditions:
\begin{equation} \label{kceiling}
H_k(n) = \sum_{i=1}^k H_k(n - (i-1)k - \sum_{t=1}^{k-1} H_k(n-(i-1)k-t))\,.\end{equation}
Once again a tree-based methodology is used to prove this result.
The infinite $k$-ary tree associated with (\ref{kceiling}) has the same skeleton as the $k$-ary tree described in the previous paragraph, but with a different labelling. This tree contains $k$ labels in each leaf and no labels in any other node. Because the labels are enumerated in preorder, it follows that this is just a sequential labelling as one traverses the leaves from left to right. Note that the resulting solution sequence $\lceil \frac{n}{k} \rfloor$ has the frequency sequence $\phi_{H_k} (v) = k$.

In what follows we use our tree-based methodology to derive and solve  a new family of $k$-ary, order $p$ recursions that includes the above two families. This family of recursions extends the arity two, $(\alpha,\beta)$-Conolly recursion (\ref{rec1}) to arity $k$. Further, we show that this family includes certain recursions whose solution sequence has frequency sequence $\gamma \phi_{H_k} + \delta \phi_{C_k}= \gamma k +\delta \phi_{C_k}$. The trees associated with these latter recursions result from the superposition of the appropriate number of copies of the trees associated with $C_k$ and $H_k$.

\subsection{The main theorem} \label{sec5:main result}

We extend the arity 2, $(\alpha,\beta)$-Conolly recursion (\ref{rec1}) to arity $k$ by extending the construction of the arity two tree from Section \ref{sec:orderp}. Recall from Section \ref{sec:orderp} that the tree constructed there relied on four parameters $s,j,m$ and $p$. For ease of exposition we limit our discussion here to the case $s=0$ and $j=1$.

Fix parameters $k \geq 3, p \geq 1$ and $m$ satisfying $p-1 \leq m \leq \frac{kp}{k-1} -1$; the infinite tree we now construct will be denoted $T = T_{m,p,k}$. The tree $T$ has the same skeleton as the infinite $k$-ary tree described above. The labelling of $T$ is as follows: the leaves of $T$ each contain $1+m$ labels. All other regular nodes each contain $x := pk - (k-1)(1+m)$ labels. The range of $m$ ensures that $x \geq 0$. The labels of the resulting tree are then enumerated in preorder. Note that for $k=2$ we get the construction of Section \ref{sec:orderp}, except that the range of $m$ is more constrained. Refer to Section 5.2, where we talk further about what goes wrong with our proof for $m<p-1$.

As before, let $T(n)$ be the subtree of $T$ with $n$ labels in preorder. If a leaf of $T(n)$ is the $i^{th}$ child of its parent at the penultimate level, then we will abbreviate it as a $i^{th}$ leaf. Define $C_T(n)$ to be the function that counts the number of nonempty leaves of $T(n)$. Then our main result is that $C_T(n)$ satisfies a $k$-ary, order $p$ nested recursion.

\begin{theorem} \label{thm:ktermorderp}
With $T$ as defined above let $C_T(n)$ be the number of nonempty leaves of $T(n)$. For $n > 2k(p+m) + p-(k-1)m$
(all labels up to the last child of the second penultimate level node must be filled), $C_T(n)$ satisfies the recursion

\begin{equation} \label{eq:ktermorderp}
R(n) = \sum_{i=1}^{k} R\left(n - (i-1)(1+m) - \sum_{t=1}^{p} R\left(n - (i-1)(1+m) - t\right)\right).
\end{equation}
\end{theorem}

\begin{remark}For $k=2$ the recursion in Theorem \ref{thm:ktermorderp} does not reduce to (\ref{rec1}). Nonetheless we will show in Section \ref{sec5:consequences} that for any $\alpha \geq 0$ there exist values of $m$ and $p$ such that the tree-based solution of (\ref{eq:ktermorderp}) is the same as that for (\ref{rec1}) .  For $\alpha < 0$ we cannot derive the corresponding solution sequence obtained through (\ref{rec1}) in Section 3 because we require $m \geq p-1$.  See Section \ref{sec5:consequences} for additional details.
\end{remark}

The proof of Theorem \ref{thm:ktermorderp} follows from the definition of $C_T(n)$ once we establish the following lemma.

\begin{lemma} \label{termcounts}
For each $1 \leq i \leq k$ the term $$C_T\left(n - (i-1)(p+m) - \sum_{t=1}^{p} C_T\left(n - (i-1)(p+m) - t\right)\right)$$
counts the number of nonempty $i^{th}$ leaves of $T(n)$.
\end{lemma}

To prove Lemma \ref{termcounts} we introduce a pruning operation on $T(n)$ for $n > k(p+m) + p-(k-1)m$ (all labels up to the first child
of the second penultimate node level must be filled). The pruned tree will be denoted $T^*(n)$.
\begin{description}
\item[initial correction step] Insert $x = pk - (k-1)(1+m)$ labels in the first supernode of $T(n)$.
\item[deletion step] Consider the subtrees $T(n-t)$ for $1 \leq t \leq p$. For each leaf of $T(n)$,
delete a label from it for every subtree $T(n-t)$ in which it is nonempty.
\item[lifting step] Lift any remaining labels from the leaves of $T(n)$ to their corresponding parent
at the penultimate level. After this step all leaves of $T(n)$ become empty.
\item[end correction step] Delete the largest $x = pk- (k-1)(1+m)$ labels that are present in $T(n)$ after the last step.
This deletion of labels offsets the insertion of $x$ labels into the first supernode during the initial
correction step.
\item[relabelling step] Delete all the (empty) leaves of $T(n)$ and relabel all remaining labels in preorder.
The previous penultimate level nodes of $T(n)$ become the new leaves.
\end{description}

The total number of labels in the pruned tree $T^*(n)$ is $n - \sum_{t=1}^p C_T(n-t)$. The key lemma follows.
\begin{lemma} \label{ktermorderpkeylemma}
The pruned tree $T^*(n)$ is identical to the tree $T(n - \sum_{t=1}^p C_T(n-t))$. Furthermore,
if $P$ is a nonempty leaf of $T^*(n)$ then the first child of $P$ in $T$ is a nonempty
leaf of $T(n)$.
\end{lemma}

It follows from Lemma \ref{ktermorderpkeylemma} that the number of nonempty first
leaves of $T(n)$ is equal to the number of nonempty leaves of $T^*(n)$. However,
the latter number is $C_T(n - \sum_{t=1}^p C_T(n-t))$ by the first assertion of Lemma \ref{ktermorderpkeylemma}.
So the assertion in Lemma \ref{termcounts} follows for $i=1$. The assertion for general $i$ follows due to the
usual bijection between $i$-th leaves of $T(n)$ and the first leaves of $T(n-(i-1)(1+m))$.
With Lemma \ref{termcounts} established it follows trivially that
$$C_T(n) = \sum_{i=1}^k C_T\left(n - (i-1)(1+m) - \sum_{t=1}^{p} C_T\left(n - (i-1)(1+m) - t\right)\right)\,.$$

We now prove Lemma \ref{ktermorderpkeylemma}. As before, we need only consider the penultimate level node
$P$ that is the last penultimate level node of $T(n)$ with a nonempty child in $T(n)$. We need to show that $P$ has
between 1 to $1+m$ labels as a leaf of the pruned tree $T^*(n)$. All other penultimate level nodes
of $T(n)$ either contain a full set of $1+m$ labels or no labels in $T^*(n)$ according to whether they have a nonempty
child in $T(n)$ or not. So suppose that $P$ is such a node and condition on the location of label $n$ in $T(n)$.

For ease of notation in the proof we let $\mu = m - p + 1$. Since $p-1 \leq m \leq \frac{k}{k-1}p -1$, we have that
$0 \leq \mu \leq p/(k-1)$. With this notation we observe that a leaf in the infinite tree $T$ contains $p+\mu$ labels
and all other regular nodes contain $x = p - (k-1)\mu$ labels. We need to show that after pruning $T(n)$ the node
$P$ contains between 1 to $p + \mu$ labels in $T^*(n)$.

\paragraph{\textbf{Case 1:}} Label $n$ is located in the $i^{th}$ child of $P$ with $1 \leq i \leq k$. Suppose that $n$ is the $l^{th}$
label in the $i^{th}$ child. If $l \leq p$ then the $i^{th}$ child loses $l-1$ labels during the deletion step and the first $i-1$
children lose $p$ labels each. If $p < l \leq p+\mu$ (assuming $\mu \geq 1$), then all $i$ children of $P$ lose $p$ labels
in the deletion step. Therefore, the number of labels in $P$ after the end correction step is $(i-1)\mu + 1$ if $l \leq p$
and $(i-1)\mu + l-p$ otherwise. Notice that
$$\min_{i,l} \left\{[(i-1)\mu + 1] \cdot \mathbf{1}_{[l \leq p]} + [(i-1)\mu + l-p]\cdot \mathbf{1}_{[l > p]}\right\} \geq 1\,.$$
This implies the second assertion of Lemma \ref{ktermorderpkeylemma} for this case. Also, using $p \geq 1$ and $0 \leq \mu \leq p/(k-1)$
it follows that
$$\max_{i,l} \left\{[(i-1)\mu + 1] \cdot \mathbf{1}_{[l \leq p]} + [(i-1)\mu + l-p] \cdot \mathbf{1}_{[l > p]}\right\} \leq p+\mu\,.$$
Thus, the first assertion of Lemma \ref{ktermorderpkeylemma} follows as well.

\paragraph{\textbf{Case 2:}} Label $n$ is the $l^{th}$ label following the final label in the last child of $P$ and $l \geq 1$.
In this case, each child of $P$ loses $p$ labels during the deletion step, and during the end correction step the number
of labels removed from $P$ is $p-(k-1)\mu - \min \{\,l, p-(k-1)\mu \,\}$. The total number of labels in $P$ after the relabelling step
is $k\mu + \min \{\,l, p-(k-1)\mu \,\}$. However, $1 \leq k\mu + \min \{\,l, p-(k-1)\mu \,\} \leq p+\mu$ due to $l, p \geq 1$ and $\mu \geq 0$.
This establishes both assertions of Lemma \ref{ktermorderpkeylemma} for this case.

The two cases above are exhaustive and Lemma \ref{ktermorderpkeylemma} is thus proved.

\subsection{Consequences of Theorem \ref{thm:ktermorderp}} \label{sec5:consequences}

We now relate our $k$-ary $(\alpha,\beta)$-Conolly generalization to our tree superposition methodology. Fix $k \geq 3$. If $(m,p) = (0,1)$ then the
solution sequence $C_T(n)$ is the $k$-ary Conolly sequence (\ref{ktermConolly}). On the other hand,
with $(m,p) = (k-1,k-1)$ the solution sequence $C_T(n) = H_k(n)$ because the resulting tree $T$
contains $k$ labels per leaf and 0 labels everywhere else. We can take the trees resulting
from these two choices of $(m,p)$ and superpose them as discussed in Section \ref{sec:linearcomb}.
To do so we fix coefficients $\gamma , \delta \geq 0$ such that at least one of them is positive, and then
set $m= k\gamma + \delta -1$ and $p = (k-1)\gamma + \delta$. With these choices it is easily verified that
$p \geq 1$ (since at least one of $\gamma$ and $\delta$ is positive), and that $p-1 \leq m \leq \frac{kp}{k-1} -1$.
The resulting tree $T_{m,p,k}$ contains $\gamma k + \delta$ labels in each leaf and $\delta$ labels in each
of the remaining regular nodes. It is the superposition of $\gamma$ copies of $T_{k-1,k-1,k}$ and $\delta$
copies of $T_{0,1,k}$. Therefore, the frequency sequence $\phi_{C_T}$ of the solution sequence $C_T(n)$ is
$\gamma \phi_{H_k(n)}  + \delta \phi_{C_k} = \gamma k + \delta \phi_{C_k}$.

Finally, because we require $m \geq p-1$ in the setup for Theorem \ref{thm:ktermorderp}, it follows from the above assignments for $m$ and $p$ that we implicitly restrict $\gamma$ to positive values. Observe that in the above $k$-ary generalization, $\gamma$ corresponds to $\alpha/2$ and $\delta$ corresponds to $\beta$ from the 2-ary case. In the latter, negative values of $\alpha$ are permitted under appropriate constraints. This leads naturally to the question whether frequency sequences with negative values of $\gamma$ can also be obtained. We comment further on this open question in the next section.
\end{section}


\begin{section}{Future directions} \label{sec:conclusion}

In the course of our investigation of nested recursions in this paper we have identified several areas where there are questions about the possibility of extending our results. In this concluding section we collect and further discuss these open problems, which provide possible directions for future research in this area.

\subsection{Tree Superposition}

Recall that in Section \ref{sec:linearcomb}, $T_{s,j,m,p}$ is defined to be the tree with the skeleton of the infinite binary tree from Figure (\ref{fig:skeleton}) with $j$ cells in each leaf and labelling scheme as follows: each of the first $j-1$ cells of each leaf receives $p$ labels, while the last cell receives $p + m$ labels. All remaining regular nodes in $T_{s,j,m,p}$ get $pj - m$ labels each, and the supernodes receive $s$ labels each. To ensure that each leaf has at least one cell and that cells have a positive number of labels, we require $p,j \geq 1$. Likewise, to force regular nodes and supernodes to contain a non-negative number of labels, we need $s \geq 0$ and $0 \leq m \leq pj$. 

For such $s,j,m,p$, Theorem \ref{thm:lincomb} states that the cell counting sequence for $T_{s,j,m,p}$ solves nested recursion (\ref{eqn3}) with sufficient number of initial conditions that follow the tree. In Section \ref{sec:linearcomb}, we observed that $T_{s,j,m,p}$ remains well-defined for negative $m$ as long as $m > -p$. However, for $-p < m < 0$, the pruning operation on $T_{s,j,m,p}(n)$ defined in Section \ref{sec:linearcomb} does not produce a tree $T^*_{s,j,m,p}(n)$ that conforms to the labelling rules described earlier, so this pruning operation can not be used to derive a nested recursion whose solution sequence is the cell counting sequence for this tree (for a specific example consider pruning $T_{0,4,-2,9}(168)$). However, this does not exclude the possibility that there exists an alternative pruning operation that does lead to suitable nested recursions in case $-p < m < 0$. This leads to our first question:

\begin{question}
Fix $s,j,m,p$ such that $p,j \geq 1$, $s \geq 0$, $-p < m < 0$. Does there exist a 2-ary order $p$ recursion of the form (\ref{eq:Conolly}) that has a solution sequence given by the cell counting sequence of $T_{s,j,m,p}$?
\end{question}

For $s = 0$ and $m=bj$ for some $b$, $T_{s,j,m,p}$ is a superposition of $T_{0,j,j}$ and $T_{0,j,0}$ from Section \ref{sec:order2} and its frequency function is a linear combination $b \phi_{H_{0,j}} + (p-b) \phi_{R_{0,j}}$. Therefore, the answer to the above question would allow to determine whether there exists a 2-ary, order $p$ recursion of the form (\ref{eq:Conolly}) whose solution sequence is a linear combination of $\phi_{H_{0,j}}$ and $\phi_{R_{0,j}}$ with negative coefficients.

\subsection{Linear Combinations of Frequency Sequences with Negative Coefficients}

In order to consider this possibility, we begin by investigating the possibility of $k$-ary, order $p$ recursions $R(n)$ with slow solutions that have frequency sequences of the form $\lambda + \delta \phi_{C_k}$, where $\lambda$ and $\delta$ are constants. In this case, some necessary conditions must be met by $\lambda$ and $\delta$.

As $R(n)$ is slow we require that $\lambda + \delta \phi_{C_k} (v) \geq 1$ for all $n$. Since $\phi_{C_k}$ is unbounded it must be the case that $\delta \geq 0$, and since $\phi_{C_k}(v) = 1$ for many values of $v$, it must be that $\lambda + \delta \geq 1$.

There is another key condition that must be met due to the asymptotic behaviour of the sequence.
Let $h_v$ be the last occurrence of $v$ in the sequence $R(n)$.
Since $R(n)$ is slow we have that $h_v = \sum_{i=1}^v \left (\lambda + \delta \phi_{C_k}(i) \right)$.
It can be easily verified that $\lim_{v \to \infty} \frac{1}{v} \sum_{n=1}^v \phi_{C_k}(i) = k/(k-1)$
because $C_k(n)/n$ converges to $(k-1)/k$ (see \cite{DR}). Therefore, $\lim_{v \to \infty} h_v / v = \lambda + \delta \frac{k}{k-1}$.
But note that $R(h_v) = v$, and so we have that $\lim_{v \to \infty} R(h_v)/h_v = \frac{k-1}{(k-1)\lambda + k \delta}$.

It is not hard to see that $R(n)/n$ must have a limit (since it is slow with frequency sequence $\lambda + \delta \phi_C$).
On the other hand, as $R(n)$ is a solution to a $k$-ary order $p$ recursion of the form (\ref{eq:Conolly}), its recursive structure
implies that any limit of $R(n)/n$ must be either zero or $\frac{k-1}{kp}$ (see, for example, \cite{ConollyLike} Theorem 2.1). In our case,
the limit can not be zero since the subsequential limit $\frac{k-1}{(k-1)\lambda + k \delta}$ is nonzero for $k \geq 2$. Thus, the
limit of $R(n)/n$ must be $\frac{k-1}{kp}$, and equating this to the subsequential limit $\frac{k-1}{(k-1)\lambda + k \delta}$ implies
that $(k-1) \lambda = k (p-\delta)$. So $k$ divides $\lambda$ since $k$ is relatively prime to $k-1$.
Therefore, $\lambda = \gamma k$ with $\gamma k + \delta \geq 1$ and $ \delta \geq 0$. This explains why in the previous section we limited our consideration to frequency sequences of the form $\gamma \phi_{H_k} + \delta \phi_{C_k}= \gamma k +\delta \phi_{C_k}$, as well as the close connection of this material to tree superpositions of the trees associated with $H_k$ and $C_k$.

The conditions $\gamma k + \delta \geq 1$ and $\delta \geq 0$ does not exclude negative values of $\gamma$.
In fact, if the frequency sequence of $R(n)$ is $\gamma k + \delta \phi_C$ with $\gamma < 0$ then $R(n)$ is the
leaf counting function $C_T(n)$ of the tree $T_{m,p,k}$ with $p = (k-1) \gamma + \delta \geq 1$ and $m = p-1+\gamma < p-1$.
Our proof of Theorem \ref{thm:ktermorderp} does not work for $m < p-1$ although the corresponding tree $T_{m,p,k}$ is
well defined. This leads to the following open problem of finding a $k$-ary, order $p$ recursion that is satisfied by $C_T(n)$ for
such $T$.

\begin{question}
Fix $\gamma < 0$ and $\delta \geq 0$ such that $\gamma k + \delta \geq 1$.
Does there exits a $k$-ary order $p$ recursion of the form (\ref{eq:Conolly}) that has
a slow solution with frequency sequence $\gamma k + \delta \phi_{C_k}$?
\end{question}

In fact, for fixed choices of $\gamma$ and $\delta$, there can be a multitude of recursions whose solution
sequences have the common frequency sequence $\gamma k + \delta \phi_{C_k}$. Classifying all such recursions is nontrivial.
Initial empirical evidence suggests that the following recursion may be a good candidate with $p = (k-1)\gamma + \delta$:
\begin{equation*}
R(n) = \sum_{i=1}^k R(n - (i-1)(p + \gamma) - R(n-1) - \sum_{t=1}^{|\gamma|} R(n-1-tk) - \sum_{t=1}^{p-|\gamma|-1} R(n-1-|\gamma|k-2t)).
\end{equation*}

\subsection{Ceiling Function Solutions to $k$-ary Order $p$ Recursions}

In \cite{ConollyLike} it is shown that $\lceil n/2p \rceil$ is the solution of ~(\ref{rec2}). In the $k$-ary recursion (\ref{eq:ktermorderp}), if we set $p = (k-1)q$ and $m = kq-1$ in (\ref{eq:ktermorderp}) it is easy to see that we get the natural generalization of this result, namely, the solution sequence is $\lceil n/kq \rceil$ (this follows because the resulting tree associated with this recursion contains $kq$ labels in each leaf, and no labels elsewhere).

Observe that if the ceiling function $\lceil \frac{an}{b} \rceil$ is a solution of any $k$-ary order $p$ recursions of the form (\ref{eq:Conolly})
with integers $a$ and $b$, then it must be that $\frac{a}{b} = \frac{k-1}{kp}$. This follows from the fact that $R(n)/n$
converges to $\frac{a}{b}$ while the recursive structure of (\ref{eq:Conolly}) implies that such a limit must be either 0 or $\frac{k-1}{kp}$
(see \cite{ConollyLike}).
This leads to the following open question that first appears in \cite{IVTCeil}:

\begin{question}
For $k \geq 3$ and $p$ not dividing $k-1$, does the ceiling function $\lceil \frac{(k-1)n}{kp} \rceil$ satisfy a $k$-ary
order $p$ recursion of the form ~(\ref{eq:Conolly})?
\end{question}

It is conjectured in \cite{IVTCeil} that no ceiling function solutions can occur unless $p = (k-1)q$.

\subsection{More Simultaneous Parameters}

In \cite{Isgur} the simultaneous parameter $q$ is introduced in the following 2-ary recursion:

\begin{equation}\label{eq:q}
R_{s,j,-q}(n) = R_{s,j,-q}(n-s-R_{s,j,-q}(n-j))+R_{s,j,-q}(n-s-j-R_{s,j,-q}(n-2j+q))
\end{equation}
where $s$ is a nonnegative integer, $j$ is a positive integer, and $q$ is an integer with $0 \leq q \leq j$. This recursion is then solved using a tree-based methodology. 

It is shown that the parameter $q$ is analogous to the parameter $m$, and as with $m$, the solution counts the number of nonempty cells in the leaves of an appropriately constructed infinite binary tree. Note the use of the negative subscript on $q$ in ($\ref{eq:q}$). This is because we want to think of recursion ($\ref{eq:q}$) as the ``negative" end of recursion ($\ref{eq:m}$); when $m=q=0$, the two coincide and are the same as recursion ($\ref{eq:0jj2j}$).
As with $(\ref{eq:m})$, extending ($\ref{eq:q}$) to $q<0$ or $q>j$ appears to never lead to well-defined solution sequences. 

Since there are strong analogies between the parameters $m$ and $q$ for certain arity 2 recursions, and since we showed in Section \ref{sec:karyorderp}
that it is possible to introduce the simultaneous parameter $m$ into arity $k$ recursions that can be solved by our tree-based methods, it is natural to ask if we can do something similar for $q$. As a possible candidate for the arity $k$ recursion family we now introduce the simultaneous parameter $j$ into the $k$-ary Conolly recursion \ref{ktermConolly} as follows:

\begin{equation}\label{eq:C(s,j,k)}
C_{s,j,k}(n) = \sum_{i=1}^k C_{s,j,k}(n-s-(i-1)j-C_{s,j,k}(n-ij))
\end{equation}

We pose the following open question:

\begin{question}
For $k \geq 3$ is it possible to find a $k$-ary family of recursions based on (\ref{eq:C(s,j,k)}) that can be solved by tree-based methods and that contain an analogue to the simultaneous parameter $q$?
\end{question}

\end{section}

\end{document}